\documentclass[11pt]{article}
	\usepackage[dvips]{graphicx}
        \usepackage[dvipsnames]{xcolor}
	\usepackage[english]{babel}
	\usepackage{hyperref}
	\usepackage{amsfonts}
	\usepackage{amssymb}
	\usepackage{amsmath}
	\usepackage{latexsym}
	\usepackage{amsthm}
	\usepackage{epsfig}
	\usepackage{mathtools}
	\usepackage{pdfpages}
        \usepackage{enumerate}
	
\newcommand{\auteur}[4]{\large #1\footnote{#2 ({\tt #3}).#4}}
\newcommand{\michele}%
	{\auteur{Michele Conforti}%
		{Dipartimento di Matematica, Universit\`a degli Studi di Padova,
			Via Trieste 63, 35121 Padova, Italy}%
		{conforti@math.unipd.it}%
		{ Supported by ``Progetto di Eccellenza 2008--2009''
			of ``Fondazione Cassa di Risparmio di Padova e Rovigo''.}}
\newcommand{\bert}%
        {\auteur{Bert Gerards}%
		{Centrum Wiskunde \& Informatica, Amsterdam, The Netherlands}%
		{bert@gerardsbase.nl}%
		{}}
\newcommand{\kostja}%
        {\auteur{Kanstantsin Pashkovich}%
		{D\'epartement de Math\'ematique, Universit\'e Libre de Bruxelles,
			Boulevard du Triomphe, B-1050 Brussels, Belgium.}%
			{kanstantsin.pashkovich@gmail.com}%
			{ Research partly done at: Dipartimento di Matematica,
				Universit\`a degli Studi di Padova, Italy.}}
\newtheorem{theorem}{Theorem}[section]
\newtheorem{lemma}[theorem]{Lemma}
\newtheorem{corollary}[theorem]{Corollary}

\newtheorem{fact}[theorem]{Fact}
\newcommand{\initial}[1]{#1^\circ}
\newcommand{\coarse}[1]{#1_{\mbox{\scriptsize\sf coarse}}}
\newcommand{\bardd}[1]{\bar{d}(#1)}
\newcommand{\dd}[1]{\bar{d}(#1)}
	\newcommand{\sm}{\backslash}
	\newcommand{\pol}[1]{{\sf P}[#1]}
	\newcommand{\neighbor}[2]{\mbox{\sf N}_{#1}{(#2)}}
	\newcommand{\buur}[2]{\mbox{\sf B}_{#1}{(#2)}}
	\newcommand{\stable}[1]{{\sf S}[#1]}
        \newcommand{\alphabound}{\alpha_{\sf bound}}
	\newcommand{\lazy}[1]{\widehat{#1}}
	\newcommand{\lazynode}[1]{\scalebox{0.8}{$\widehat{#1}$}}
        \newcommand{\lazynodeA}[2]{\lazynode{#2}_{#1}}
	\newcommand{\record}[1]{#1^{\mbox{\tiny\sf record}}}
	\newcommand{\recordnode}[1]{\powernode{#1}}
	\newcommand{\lift}[1]{\record{#1}\sm \{\recordnode{\emptyset}\}}
	\newcommand{\liftnode}[1]{\powernode{#1}}
	\newcommand{\gaA}[1]{\mathcal A_{#1}}
	\newcommand{\power}[1]{#1^{\mbox{\tiny\sf power}}}
	\newcommand{\powernode}[1]{r_{#1}}
	\newcommand{\connect}[1]{#1^{\mbox{\tiny\sf connect}}}
	\newcommand{\Hom}[1]{\operatorname{\sf hom}_#1}
	\newcommand{\Correct}{\operatorname{\sf correct}}
	\newcommand{\correct}[1]{\Correct_{#1}}
	\newcommand{\Service}{\operatorname{\sf service}}
	\newcommand{\load}[1]{\operatorname{\sf load}#1}
	\newcommand{\Load}[1]{\operatorname{\sf load}(#1_{> 0})}
	\newcommand{\pos}[1]{|#1_{>0}|}
	\newcommand{\rootcount}[1]{\operatorname{\sf root-size}(#1_{> 0})}
        \newcommand{\rootservant}{(G_2,\lazy{\mathcal U})}%
        \newcommand{\rootindmaster}{(G_1, Z)}%
\newcommand{\R}{\mathbb R}  \newcommand{\Q}{\mathbb Q}
\newcommand{\cA}{\mathcal A} \newcommand{\cB}{\mathcal B} \newcommand{\cC}{\mathcal C}
  
\newcommand{\cG}{\mathcal G} \newcommand{\cH}{\mathcal H} 
  \newcommand{\cL}{\mathcal L}
  
\newcommand{\cP}{\mathcal P} \newcommand{\cQ}{\mathcal Q} \newcommand{\cR}{\mathcal R}
  \newcommand{\cU}{\mathcal U}
 \newcommand{\cW}{\mathcal W} \newcommand{\cX}{\mathcal X}
 
\begin{document}
\thispagestyle{empty}
\begin{center}\LARGE Stable Sets and Graphs with no Even Holes
\\[\baselineskip] \michele, \bert, \kostja \\[\baselineskip] \today
\end{center}
\begin{abstract}
We develop decomposition/composition tools
for efficiently solving maximum weight
stable sets problems as well as for describing them as polynomially sized linear
programs (using ``compact systems''). Some of these are well-known but need
some extra work to yield polynomial ``decomposition schemes''.

We apply the tools to graphs with no even hole and no cap. A hole is a
chordless cycle of length greater than three and a cap is a hole together with
an additional node that is adjacent to two
adjacent nodes of the hole and that has no other neighbors on the hole.
\end{abstract}
\section{Introduction}\label{intro}
A vast literature about efficiently solvable cases of the stable set problem
focuses on ``perfect graphs''. Based on the ellipsoid method, Gr\"otschel,
Lov\'asz, and Schrijver  \cite{GLS} have developed a polynomial-time algorithm that
computes a stable set of maximum weight in a perfect graph. Perfect graphs have no
odd holes. (A {\em hole} is a chordless cycle of length greater than three.)
It is conceivable that the stable set problem is polynomially solvable
for all graphs without odd holes, and this may even extend to graphs with all holes having
the same parity, so either all even or all odd.
To our knowledge the case that all holes are odd has not received much
attention and in this paper we take a first step in exploring this topic 
by considering ``cap-free'' graphs with no even holes.
A {\em cap} is a hole together with an additional node that is adjacent to two
adjacent nodes of the hole and that has no other neighbors on the hole.

\begin{theorem}\label{algtheorem}
The stable set problem for cap-free graphs with no even holes is polynomially solvable.
\end{theorem}
The {\em stable set polytope} of a graph is the convex hull
of the characteristic vectors of the stable sets of the graph. 
Linear descriptions of stable set polytopes require in the worst case exponentially
many inequalities and arbitrarily large coefficients (in minimum integer form),
even for cap-free graphs with no even hole.
However, for those graphs we can tame the descriptions by allowing some extra
variables.
An {\em extended formulation} for a polytope $P$ in $\R^n$ is a system of inequalities
$Ax+By\le d$ such that
$$P=\{x\in \R^n:\exists y\; [Ax+By\le d]\;\}.$$
An extended formulation for $P$ is {\em compact} if its encoding has polynomial size in $n$.

\begin{theorem}\label{compacttheorem}
Stable set polytopes
of cap-free graphs without even hole admit compact extended formulations.
\end{theorem}
\noindent
We develop decomposition/composition tools for solving maximum weight
stable sets problems.
The working of such tool is that
when a graph is decomposable into smaller parts according
to the tool's specifications,
then that can be used to efficiently construct a solution for the whole from
solutions for the parts. Some of these tools are well-known but need
some extra work to make them suitable as a component in polynomial-time algorithms.
We develop similar mechanisms for combining polynomially sized linear
programs for stable set problems on parts of a decomposition into
such linear program for the whole.
\\[.5\baselineskip]
We apply these results to cap-free graphs with no even holes.
Conforti, Cornu\'{e}jols, Kapoor, and Vu\v{s}kovi\'c
\cite{CAP,TFODD,OSdec}
give a decomposition theorem for graphs with no even hole and use that to find
even holes in polynomial time \cite{OSrec}.
The following theorem is a simplified variant of the main result in \cite{CAP}.
\begin{theorem}[{\cite[Theorem 4.1]{CAP}}]
\label{capdec}
Every cap-free graph with a triangle either admits an 
amalgam decomposition or a clique cutset decomposition {\em (both defined in Section 2)} or contains a node adjacent to all other nodes. 
\end{theorem}
\noindent
So cap-free graphs with no even holes can be built from triangle-free graphs with no
even holes.
Conforti, Cornu\'{e}jols, Kapoor, and Vu\v{s}kovi\'c
\cite{TFODD} prove that triangle-free graphs with no
even holes can be further decomposed into as simple graphs as ``fans''
and the 1-skeleton of the three dimensional cube.
This is Theorem \ref{tf}; as that result is a bit technical, we explain its details later,
in Section~\ref{pfthm1}.
As we will see, all decompositions coming up in Theorems~\ref{capdec} and~\ref{tf} fall
in our framework and thus, taking all together, we get Theorems~\ref{algtheorem}
and~\ref{compacttheorem}.
\\[.5\baselineskip]
{\bf Notation.}
A {\em transversal} of a collection $\cA$ of disjoint nonempty sets
is a set $W\subseteq \cup\cA$ with $|W\cap X|=1$ for each $X\in \cA$.

Let $G=(V,E)$ be a graph.
If $X,Y\subseteq V$ are disjoint and some node in $X$ has a neighbor in $Y$,
then $X$ and $Y$ are {\em adjacent}.
If each node in $X$ is adjacent to all nodes in $Y$, then
$X$ is {\em fully adjacent} to $Y$.
The set of nodes outside $X$ that are adjacent to $X$ is denoted by 
$\neighbor{G}{X}$.
The subgraph of $G$ induced by $X\subseteq V$ is $G_X$.
Moreover, $G-X=G_{V\sm X}$ and $\buur{G}{X}=\neighbor{G}{V\sm X}$.
If $x\in V$, we write $G-x$ for $G-\{x\}$, $\neighbor{G}{x}$ for $\neighbor{G}{\{x\}}$,
``$x$ is fully adjacent to $Y$'' for ``$\{x\}$ is fully adjacent to $Y$'',
etc.

Let $\cA$ be a collection of disjoint nonempty sets in $G$.
Let $G_{\cA}$ denote the graph with as nodes the members of $\cA$
and as edges the adjacent pairs in $\cA$.
We call any graph isomorphic to $G_{\cA}$ a {\em  pattern} of $\cA$.
We call $\cA$ a {\em region} of $G$,
if each adjacent pair in $\cA$ is fully adjacent;
in that case, 
$G_W$ is a pattern of $\cA$ in $G$,
for every transversal $W$ of $\cA$.

If $X$ is fully adjacent to $\neighbor{G}{X}$ and $X\neq\emptyset$,
we call $X$ a {\em group of $G$}.
A partition of a set $A$ into groups
of $G$ is a {\em grouping} of $A$ in $G$.

The collection of stable sets in a graph $G=(V,E)$ is denoted by $\stable{G}$.
The {\em stability number} $\alpha(G)$ of $G$ is
the size of the largest stable set in $G$.
If $X\subseteq V$, then $\alpha(X)=\alpha(G_X)$.
If $w=(w_v:v\in V)\in\R^V$, then $w_X=(w_v:v\in X)$ and
$w(X)=\Sigma_{v\in X}w_v$.
\section{Solving the stable set problem by decomposition}\label{decomposition}
Given a graph $G=(V,E)$ and a weighting $(w_v:v \in V)$,
we consider the following problem:
\\[.5\baselineskip]
\indent
Find in $G$, a stable set $S$, that maximizes $w(S)$.
\hfill({\em Stable set problem})
\\[.5\baselineskip]
Suppose we are given a triple $(V_1,U,V_2)$ of disjoint sets with union $V$ 
and a grouping $\cU$ of $U$
in $G_{V_2\cup U}$ such that
$V_1$ and $V_2$ nonadjacent, $|V_1\cup U|>|\cU|$ and $|V_2|>0$.
We call such $(V_1,U,V_2)$ with $\cU$ a {\em node cutset separation} of $G$.
(Recall from Section~\ref{intro}, that $\cU$ is a grouping means that
each $X\in \cU$ is fully adjacent to $N_G(X)\cap V_2$ and
that each pair in $\cU$ is either nonadjacent or fully adjacent.)

Define for each stable set $S'$ in $V_1\cup U$ the
value:
$$\Correct(S')=\max\{w(S'')\,:\, S''\subseteq V_2, S'\cup S''\in\stable{G}\}.$$
Then the maximum $w(S)$ of a stable set $S$ in $G$ is
equal to the maximum in the following problem.
\\[.5\baselineskip]
\indent
Find in $G_{V_1\cup U}$, a stable set $S'$,
that maximizes $w(S')+ \Correct(S')$.
\hfill({\em Master})
\\[.5\baselineskip]
Fix a transversal \ $\lazy\cU$ of $\cU$.
It is immaterial which particular transversal is chosen;
our actual object of interest is 
the graph $G_{V_2\cup\, \lazy{\cU}}$ and
(up to graph isomorphism) that does not depend on the choice of
$\lazy\cU$---that is what ``$\cU$ is a grouping of $U$ in $G_{V_2\cup U}$'' means.
We define for each $\cB\subseteq \cU$,
the set $\lazy{\cB}=\{\lazynode{X}\,:\,X\in \cB\}$ and
we define the map $\Hom\cU\!\!: \stable{G_{V_1\cup U}} \rightarrow 
\stable{G_{\lazy\cU}}$ by:
$$\Hom{\cU}(S')=\{\lazynode X\,:\, X\in \cU, S'\cap X\neq \emptyset\}$$
for each stable set $S'$ in $V_1\cup U$.

Now take any stable set $S'$ in $V_1\cup U$.
Then, for any set $S''$ in $V_2$, it is straightforward to see that 
$S'\cup S''$ is a stable set in $G$ if and only if 
$\Hom{\cU}(S')\cup S''$ is a stable set in $G_{V_2\cup\, \lazy{\cU}}$.
This means that
$$\Correct(S')=\Service_{\,\lazy\cU}(\Hom{\cU}(S')),$$
where the function $T \mapsto \Service_{\lazy\cU}(T)$ on $\stable{\lazy\cU}$
is given by the values $w(V_2\cap S_T)$ of the stable sets $S_T$ determined by:
\\[.5\baselineskip]
\indent
Find in $G_{V_2\cup\, \lazy{\cU}}$,
for each stable set $T$ in $\lazy\cU$,\\
\indent
\indent
a stable set $S_T$ with $\lazy{\cU}\cap S_T=T$,
that maximizes $w(V_2\cap S_T)$.\hfill({\em Servant})
\\[.5\baselineskip]
The discussion above implies the following results.
\begin{gather}\label{maxformula}
\max_{S\in\stable{G}}w(S)=\max_{S'\in \stable{G_{V_1\cup U}}}w(S')+ \Correct(S'),\\
\intertext{
where\ $\Correct=\Service_{\,\lazy\cU}\circ \Hom{\cU}$ and
}
\label{servantformula}
\Service_{\,\lazy\cU}(T)=
\max_{S_T\in \stable{G_{V_2\cup\,\lazy\cU}},\, \lazy\cU\cap S_T=T}w(V_2\cap S_T).
\end{gather}
\begin{fact}\label{formula}
Suppose we are given:
\begin{enumerate}[(1)]
\itemsep 0mm
\item
a solution $(S_T\,:\,T\in \stable{\lazy{\cU}})$ of the servant,
\item
the function $\Service_{\,\lazy\cU}: T\mapsto w(V_2\cap S_T)$ on $\stable{\lazy{\cU}}$,
\item
a solution $S'$ of the master with 
$\Correct=\Service_{\,\lazy\cU}\circ \Hom{\cU}$,
\item
$S=S'\cup (V_2\cap S_{\Hom{\cU}(S')})$.
\end{enumerate}
Then $S$ is a solution of the stable set problem on $G$ with $w$.
\end{fact}
\noindent
Fact~\ref{formula} and the underlying formula (\ref{maxformula}) are
seminal to the approach in this paper.
It says that
we can solve the stable set problem on $G$ as follows:
first, solve all the problems listed in the servant and
substitute the results in the objective function of the master;
after that, solve the master.
We call this {\em master/servant decomposition}, and use that term freely at all
levels: for instance, we will call
$G_{V_1\cup U}$, $G_{V_2\cup\,\lazy\cU}$ a master/servant decomposition
of $G$.
\\[.5\baselineskip]
Singletons are groups.
The advantage of using larger groups is that that takes out
replications in the list of problems making up the servant.
For us that saving is crucial: The decompositions in Theorem~\ref{capdec}
come from node cutset separations with unbounded $\alpha(U)$
but with $\alpha(\lazy{\cU})=1$; and that is a big difference:
the servant comprises as many problems as there are stable sets in $\lazy\cU$,
and that number lies
between $2^{\alpha(\lazy\cU)}$ and
${|\lazy\cU|\choose\alpha(\lazy\cU)}2^{\alpha(\lazy\cU)}$.
Our algorithms will all come with an a priori bound $\alphabound$
on $\alpha(\lazy{\cU})$, but allow $\alpha(U)$ to be arbitrary high.

Observe that for any node $v\in U$, 
the union of all groups in $U$ that contain $v$ is a group.
So the inclusion-wise maximal groups in $U$ form a grouping of $U$,
we denote this grouping by $\coarse{U}$.
It is clearly ``better than the rest'':
the pattern $G_{\coarse{\lazy U}}$ is a proper
induced subgraphs of the pattern $G_{\lazy{\cU}}$ for any other $\cU$.
Incidentally, note that it is not hard to find 
the maximal group in $U$ containing a particular node $v$:
starting with $X=U$,
keep removing nodes from $X$ that have not the same neighbors
outside the current $X$ as $v$
until this is no longer possible.
Then, $X$ has become a group. 
As all groups in $U$ containing $v$ will have stayed in $X$ during the procedure,
$X$ is the maximal group in $U$ that contains $v$.
\\[.5\baselineskip]
{\bf Further decomposing the servant---rooted graphs.}
Our approach will be to not only apply master/servant decomposition to
the stable set problem on $G$ but also to the problems that appear in the master and the
servant. This is not without issues, both for the master as the servant.
We first consider the servant.

The servant is a {\em stable set problem on a rooted graph}; which in
general is formulated as:
Given a graph $G=(V,E)$, a {\em root} $Z\subseteq V$, and a weighting $(w_v:v \in V)$:
\\[.5\baselineskip]
\indent
Find in $G$, for each stable set $R$ in $Z$,\\
\indent
\indent
a stable set $S_R$ with $Z\cap S_R=R$, that maximizes $w(S_R)$.\\
\mbox{}\hfill({\em Stable set problem with root})
\\[.5\baselineskip]
The servant is special in that it has weighting $0$ on the root
(because $w(S_R)=w(S_R\sm Z)$ if $w$ is identical to $0$ on $Z$).
We call a pair $(G,Z)$ with $Z\subseteq V$ a {\em rooted} graph.
We call $V\sm Z$ the {\em area} of $(G,Z)$ and of the stable
set problem on $(G,Z)$;
if the area is empty we call $(G,Z)$, and
the stable set problem, {\em trivial}.
The servant is never trivial.

The stable set problem on a rooted graph is a list of stable set problems.
That the servant has this ``multiple-problem'' feature becomes an
issue when we further decompose these subgraphs of the servant as
if they were totally unrelated.
We easily run into exponential explosion then;
even if the root has as few as two stable sets,
which is always the case when $U$ is nonempty.
That means that we can hardly iterate decomposing ``on the servant side''.

We ran into this issue of the ``multiple-problem'' aspect of the servant
when we wanted to use amalgam decompositions in Theorem~\ref{capdec} to design
a polynomial time algorithm for the stable set problem of cap-free graphs
with no even holes.
A standard way to address the issue is to avoid servant-graphs that can be
further decomposed, for instance by taking $V_2$ inclusion-wise minimal.
Clique cutsets with $V_2$ inclusion-wise minimal
have that property and can be found efficiently (Whitesides~\cite{WHI1}).
But for the amalgam decompositions used in Theorem~\ref{capdec},
this will not work.
Cornu\'ejols and Cunningham \cite{CC}
gave a polynomial-time algorithm to find an amalgam separation with
minimal servant, but as illustrated by Figure~\ref{hiddenamalgam},
that does not guarantee that the amalgam blocks will have no amalgams. 
So, what then?
Just forbidding to decompose ``on the servant side'' and ignore occasions
that arise limits the applicability of the approach too much---at
least for our purposes.

There is a way out:
master/servant decomposition extends easily to rooted graphs
$(G,Z)$ with $Z\subseteq V_1\cup U$:
in the master, just replace the graph $G_{V_1\cup U}$ by 
the rooted graph $(G_{V_1\cup U},Z)$, but keep the same objective function
$S\mapsto w(S)+\Correct(S)$ with
$\Correct=\Service_{\,\lazy\cU}\circ \Hom{\cU}$, where
$\lazy\cU$ is the same transversal of $\cU$ and
$\Service_{\,\lazy\cU}$ comes from the same servant, with the same graph,
the same root and the same weighting as before.
\\[.5\baselineskip]
\indent
Find in $G_{V_1\cup U}$, for each stable set $R$ in $Z$,\\
\indent
\indent
a stable set $S'_R$ with $Z\cap S'_R=R$, that maximizes $w(S'_R)+ \Correct(S'_R)$.\\
\mbox{}
\hfill({\em Master with root})
\\[.5\baselineskip]
In terms of rooted graphs, (\ref{maxformula}) reads as:
\begin{eqnarray}\label{rootmaxformula}
\max_{S_R\in\stable{G},\,Z\cap S_R=R}w(S_R)=
\max_{S'_R\in \stable{G_{V_1\cup U}},\,Z\cap S'_R=R}w(S'_R)+ \Correct(S'_R).
\end{eqnarray}
with the function $\Correct=\Service_{\,\lazy\cU}\circ \Hom{\cU}$ is as in (\ref{servantformula}).
Moreover:
\begin{fact}\label{size}
The area of $(G,Z)$ is the disjoint union of the area of the master and
the servant.
\end{fact}
\noindent
This is crucial to our approach.
The stable set problem on $G$ is the same as
the stable set problem on the rooted graph $(G,\emptyset)$.
Starting from the rooted graph perspective,
we reduce a stable set problem on a rooted graph into one stable set problem
on the rooted master-graph and one stable set problem
on the rooted servant-graph.
Regardless how often we repeat this master/servant decomposition for rooted graphs,
by (\ref{size}),
the collective area of the list of problems constructed does not grow.
We formalize this by ``decomposition lists''.
\\[.5\baselineskip]
Suppose $(G,Z)$ occurs
in an left/right ordered list $\mathcal L$ of rooted graphs.
Then a {\em master/servant decomposition of $\mathcal L$
along node cutset separation $(V_1,U,V_2)$ with $\cU$ of $(G,Z)$} is
a list obtained from $\cL$ by replacing $(G,Z)$ by 
the rooted master graph $(G_{V_1\cup U},Z)$ (in the same position) and then inserting
the rooted servant graph $(G_{V_2\cup\, \lazy{\mathcal U}},\lazy{\mathcal U})$
anywhere further down the list (so to the right of where $(G_{V_1\cup U},Z)$ is).

A {\em master/servant decomposition-list} of $(G,Z)$ is either
$((G_,Z))$, or (recursively) defined as a master/servant decomposition
of a master/servant decomposition list of $(G,Z)$.
\begin{lemma}\label{graphcount}
A master/servant decomposition list of a rooted graph with $n$ nodes
contains at most $n$ non-trivial rooted graphs and at most $2n^2$ trivial rooted graphs.
\end{lemma}
\begin{proof}
Imagine a sequence of master/servant
decompositions that leads from the rooted graph to the list.
Trivial rooted graphs can not be decomposed, so for this
analysis we ignore them.
We visualize each non-trivial rooted graph encountered
in the sequence as an area-root pair with nonempty area and root of
size at most $n-1$. These area-root pairs behave as follows.
Either we split an area into two disjoint nonempty areas and
assign each of them with a root, each with at most $n-1$ nodes,
or we reduce the root of an area-root pair to a proper subset without
changing the area of that pair.

As the initial area has size at most $n$,
we can apply at most $n-1$ ``area-splits''.
So overall, we will encounter at most $2n-1$ different areas.
Since roots have size at most $n-1$,
the ``root-reduction''
operation cannot be iterated more than $n-1$ times without changing the area.
\end{proof}
\noindent
The rooted graph approach fully overcomes 
the issues of the multiple-problem aspect of the servant that 
we ran into when we wanted to use Theorem~\ref{capdec} for
designing algorithms.
\\[.5\baselineskip]
{\bf Further decomposing the master---templates.}
Suppose the master-graph $G_{V_1\cup U}$ has a
node cutset separation $(V'_1,U',V'_2)$ with grouping $\cU'$
and $Z\subseteq V'_1\cup U'$ and we want to use that for a 
master/servant decomposition of the master.
Since the factor $\Service_{\,\lazy\cU}$ of the extra term
can be virtually anything (see Fact~\ref{servantrecord}),
we only do that if the separation {\em fits} $\cU$, which means that:
\begin{enumerate}[(i)]
\itemsep=0cm
\item
$U\subseteq V'_1\cup U'$ or  $U\subseteq V_2'\cup U'$.
\item
if $U$ meets $V'_2$ and $X\in \cU$ meets $U'$,  then
$X\cap U'$ is the union of members of $\cU'$.
\end{enumerate}
When the separation does fit $\cU$, we decompose the master as follows:
The master-of-the-master has graph $G_{V'_1\cup U'}$, root $Z$, and
objective function:
\\[.5\baselineskip]
\mbox{}\hfill
$(\Service_{\,\lazy\cU'}\circ {\Hom{{\cU'}}})(S'_R)\ +\ w(S'_R)\ +
\begin{cases}
\ (\Service_{\,\lazy\cU}\circ{\Hom{\cU}})(S'_R)&\ \text{if $U$ does not meet $V'_2$,}\\
\hfill 0\hfill &\ \text{if $U$ does meet $V'_2$,}
\end{cases}$\hspace*{7mm}\mbox{}
\\[.5\baselineskip]
and the servant-of-the-master has graph $G_{V'_2\cup\,\lazy{\cU'}}$, root $\lazy{\cU'}$,
and objective function:
\\[.5\baselineskip]
\mbox{}\hfill
$w(V'_2\cap S_T)\ +\ 
\begin{cases}
\hfill 0\hfill &\ \text{if $U$ does not meet $V'_2$,}\\
\ (\Service_{\,\lazy\cU}\circ{\Hom{\cU}})(S'_R)&\ \text{if $U$ does meet $V'_2$}.
\end{cases}$\hspace*{7mm}
\\[.5\baselineskip]
If we next also decompose the master-of-the-master and
servant-of-the-master and keep repeating that, we will accumulate more and more
extra terms of the form $\Service_{\,\lazy\cA}\circ{\Hom{\cA}}$.
This leads to {\em templates}: triples $(G,Z,\Omega)$ where $(G,Z)$ is a rooted
graph and $\Omega$ a collection of regions in $G$.
Also note that the presence of these extra terms has no effect on how rooted graphs
are decomposed; the extra terms only prevent some separations
to lead to decompositions.
This means that (\ref{size}) and Lemma~\ref{graphcount} will still apply.
\\[.5\baselineskip]
Suppose we are given a collection $\Omega$ of regions in $G$ and
for each of regions $\cA\in \Omega$: a pattern with node set $\lazy\cA$,
a graph isomorphism $X \rightarrow \lazynodeA{\cA}{X}$ between
$G_\cA$ and that pattern,
and a real-valued function $\sigma_{\lazy\cA}$ on $\stable{\lazy\cA}$.
Define, for $\cB\subseteq \cA$,
the set $\lazy{\cB}_{\cA}=\{\lazynodeA{\cA}{X}\,:\,X\in \cB\}$,
and for $S\in \stable{G}$, the set
${\Hom{\cA}}(S)=\{\lazynodeA{\cA}{X}\,:\, X\in \cA, S'\cap X\neq \emptyset\}$.
We consider the following problem.
\\[.5\baselineskip]
\indent
Find in $G$, for each stable set $R$ in $Z$,\\
\indent
\indent
a stable set $S_R$ with $Z\cap S_R=R$, that maximizes
$\displaystyle{w(S_R)+
\sum_{\cA\in\Omega} (\sigma_{\lazy\cA}\circ {\Hom{\cA}})(S_R)}.$
\\[.1\baselineskip]
\mbox{}\hfill
({\em Stable set problem on a template})
\\[.5\baselineskip]
Suppose also that our node cutset separation $(V_1,U,V_2)$ with $\cU$ and
$Z\subseteq V_1\cup U$ fits each region in $\Omega$.
Then the maxima in the stable set problem on a template are the same as the
maxima in the following problem.
\\[.5\baselineskip]
\indent
Find in $G_{V_1\cup U}$, for each stable set $R$ in $Z$,\\
\indent
\indent
a stable set $S'_R$ with $Z\cap S'_R=R$, that maximizes\\
\mbox{}\hfill
$\displaystyle w(S'_R)+ \sum_{\cA\in\Omega, V_2\cap (\cup\cA)=\emptyset}
(\sigma_{\lazy\cA}\circ {\Hom{\cA}})(S'_R)\ +\ \correct{}(S'_R)$,\hspace*{7mm}
\\[-.2\baselineskip]\mbox{}\hfill
({\em Master for a template})
\\[.5\baselineskip]
where $\correct{}=\Service_{\lazy\cU}\circ{\Hom{\cU}}$ and
$T \mapsto \Service_{\lazy\cU}(T)$ on $\stable{\lazy\cU}$
is given by the maxima in:
\\[.5\baselineskip]
\indent
Find in $G_{V_2\cup\lazy\cU}$, for each stable set $T$ in $\lazy\cU$,\\
\indent
\indent
a stable set $S_T$ with $S_T\cap Z=T$, that maximizes\\
\mbox{}\hfill
$\displaystyle
w(V_2\cap S_T)+
\sum_{\cA\in\Omega, V_2\cap (\cup\cA)\neq\emptyset} (\sigma_{\lazy\cA}\circ {\Hom{\cA}})(S_T).$\hspace*{7mm}
\\[.1\baselineskip]\mbox{}\hfill
({\em Servant for a template})
\\[.5\baselineskip]
Using templates, formula (\ref{maxformula}) extends to:
\begin{gather}
\max_{S_R \in \stable{G}, Z\cap S_R=R}
{w(S_R)\ +\sum_{\cA\in\Omega} (\sigma_{\lazy\cA}\circ {\Hom{\cA}})(S_R)}=\hspace*{6cm}\nonumber \\
=
\ \ \max_{S'_R\in \stable{G_{V_1\cup U}}, Z\cap S'_R=R}
w(S'_R)\ + \sum_{\cA\in\Omega, V_2\cap (\cup\cA)=\emptyset}
(\sigma_{\lazy\cA}\circ {\Hom{\cA}})(S'_R)+ \Correct{}(S'_R),\\
\intertext{where\ $\Correct{}=\Service_{\lazy\cU}\circ {\Hom{\cU}}$ and}
\label{tempservant}
\max_{S_R \in \stable{G},\, Z\cap S_R=R}
\Service_{\,\lazy\cU}(S_T)=
w(V_2\cap S_T)\ + \sum_{\cA\in\Omega, V_2\cap (\cup\cA)\neq\emptyset}
(\sigma_{\lazy\cA}\circ {\Hom{\cA}})(S_T).
\end{gather}
Also Fact~\ref{formula} generalizes to templates.
\begin{fact}\label{rootformula}
Suppose we are given:
\begin{enumerate}[(1)]
\itemsep 0mm
\item
a solution $(S_T\,:\,T\in \stable{\lazy\cU})$ of the servant for a template,
\item
the function
$\Service_{\lazy\cU}: T\mapsto
w(V_2\cap S_T)+
\sum_{\cA\in\Omega,\ V_2\cap (\cup\cA)\neq\emptyset}
(\sigma_{\lazy\cA}\circ {\Hom{\cA}})(S_T)$
on $\stable{\lazy{\cU}}$,
\item
a solution $(S'_R\,:\,R\in \stable{Z})$ of the master for a template with
$\Correct=\Service_{\lazy\cU}\circ {\Hom{\cU}}$,
\item
$S_R=S'_R\cup (V_2\cap S_{{\Hom{\cU}}(S'_R})$ for each $R\in \stable{Z}$.
\end{enumerate}
Then
$(S_R \,:\,R\in \stable{Z})$ is a solution of the stable set problem on a template.
\end{fact}
\noindent
And we have the following.
\begin{fact}
The master for a template and the servant for a template are both
stable set problems on a template.
\end{fact}
\begin{proof}
For the master this is obvious:
each region that does not meet $V_2$---and that includes also $\cU$---is
a region in the master-graph.
For the servant the situation is slightly subtle.
The regions in $\Omega$ that meet $V_2$ are regions
of $G_{V_2\cup U}$, but not of $G_{V_2\cup\, \lazy\cU}$.
However, if we replace the members of any such region
by their intersection with $V_2\cup\, \lazy\cU$,
we do get a region of $G_{V_2\cup\, \lazy\cU}$;
for the formula in (\ref{tempservant}) that replacement has no effect.
\end{proof}
\noindent
This extension of the master/servant decomposition
to stable set problems on templates ``closes'' our model.
\begin{theorem}\label{th:templ-stablesets}
Let $\cR$ be a family of templates closed under master/servant decomposition and
$\cP$ be a subfamily of $\cR$.
If there is a polynomial time that finds a fitting node cutset separation
for any template from $\cR\sm\cP$
and there is a polynomial time algorithm
for the stable set problem on templates from $\cP$,
then there exists a polynomial time algorithm for the stable set problem on templates from $\cR$.
\end{theorem}
\begin{proof}
Suppose we are given a stable set problem on a 
template $(\initial G,\initial Z,\initial \Omega)\in\cR$.
To solve it, we keep a right/left ordered list of stable set problems on
templates $(G,Z,\Omega)$ such that the underlying list of rooted graphs
$(G,Z)$ is a master/servant decomposition list
of rooted graphs. Initially this list consists of just the single 
stable set problem on $(\initial G,\initial Z,\initial \Omega)$.

For any template $(G,Z,\Omega)$ on the list,
the regions $\cA\in \Omega$ all come with a real valued function
$\sigma_{\lazy\cA}$ on $\stable{\lazy\cA}$, which is either given by
an explicit listing of all function values $\sigma_{\lazy\cA}(T)$
with $T\in \stable{\lazy\cA}$ or by 
$\sigma_{\lazy\cA}=\Service_{\lazy\cA}$, where
the function values of $\Service_{\lazy\cA}$ are the
maxima of a stable set problem with
template $(G',\lazy\cA,\Omega')$ further down the list
(so to the right of where $(G,Z,\Omega)$ is). As soon
as the solution of that stable set problem on $(G',\lazy\cA,\Omega')$
comes available, we store the solution by an explicit listing of the values
$\Service_{\lazy\cA}(T)$ and remove
the stable set problem on $(G',\lazy\cA,\Omega')$ from the list.
As of that moment $\sigma_{\lazy\cA}$ is represented by the
stored explicit listing.
To solve that stable set problem on $(G',\lazy\cA,\Omega')$,
we need an explicit listing for each $\sigma_{\lazy\cA'}$ with
$\cA'\in\Omega'$. The right most problem on the list has that property.
Therefore, we always try find a solution to that right most problem.

So the algorithm is to iterate the following procedure:
Remove the right most stable set problem from the list
and either decompose that problem and place the master and servant
in that order at the end of the list, or solve the removed problem and store its
solution.
By the given decomposition algorithm for $\cR\sm\cP$ and the given optimization
algorithm for $\cP$ we can carry out each iteration in 
polynomial time.
By Lemma~\ref{graphcount} we can iterate the procedure at most $2n^2+n$
times if $\initial G$ has $n$ nodes.
So after at most $2n^2+n$ iterations the list is empty.
That means that we stored an explicit listing of the solution of the first
problem in the list. Since throughout
the entire algorithm the first problem keeps the
initial root $\initial Z$, that explicit listing solves the original
stable set problem on 
$(\initial G,\initial Z,\initial \Omega)$.
\end{proof}
\noindent
{\bf Linearized decomposition.}
Instead of carrying around the nonlinear terms $\sigma_{\lazy\cA}\circ\Hom{\cA}$, we can 
``linearize'' them by adding suitably weighted nodes that are adjacent to nodes in $\cup \cA$.
We are free in choosing which terms  $\sigma_{\lazy\cA}\circ\Hom{\cA}$ we eliminate in this way and when we do that.
Our algorithms in this paper either not use the option at all or do it at each decomposition at once, as part of the decomposition procedure.
\\[.5\baselineskip]
We say that a triple $[H,\gamma,\sigma]$ {\em linearizes}
a real-valued function $d$ on $\stable{A}$,
if $H=(W,F)$ is a graph with $A\subseteq W$,
$\gamma\in\R^W$, $\sigma\in\R$ such that
for each $T\in \stable{A}$,
the maximum value $\gamma(S_T)$ of stable set in $H$ with $A\cap S_T=T$
is equal to $d(T)-\sigma$.
\\[.5\baselineskip]
To see the relevance of this definition,
suppose that $W\cap V=U$ and that
$[H,\gamma,\sigma]$ linearizes $\sigma_{\lazy\cU}\circ {\Hom{\cU}}$.
Consider the graph $G_{V_1\cup U}\cup H=(V_1 \cup W,E\cup F)$.
Define $\gamma_v=0$ if $v\in V\sm W$.
\begin{fact}
$S^+\in\stable{G_{V_1\cup U}\cup H}$ with $Z\cap S^+=R$ maximizes
$w(V\cap S^+)+\gamma(W\cap S^+)$, then 
$S_R=V \cap S^+$ is a stable set in $G$ with $Z\cap S_R=R$ that maximizes
$w(S_R)+ (\sigma_{\lazy\cU}\circ \Hom{\cU})(S_R).$
\end{fact}
\noindent
So we can use a triple $[H,\gamma,\sigma]$ linearizing $\sigma_{\lazy\cU}\circ {\Hom{\cU}}$
to reformulate the master problem so that the objective function is linear.
This linearization comes with the expense of adding nodes (unless $W=U$).
\\[.5\baselineskip]
A canonical way to linearize functions on $\stable{A}$
with $A\subseteq V$ is to ``add a record of $\stable{A}$ to $G_A$''.
{\em Adding a record of $\stable{A}$} to a graph $G$, means
to add a clique $\record{A}$ (the {\em record\,})
consisting of new nodes $\recordnode{T}$, one for each $T\in \stable{A}$,
such that each $\recordnode{T}$ is fully adjacent to $A\sm T$ and nonadjacent
to $T$ and to $V\sm A$.
We call the new graph the {\em record graph} of $G$ and $A$,
and we  denote it by $G(A)$. Records play a major role in this paper. The following fact only uses
the record graph $G_A(A)$ of $G_A$ and $A$.
\begin{fact}\label{record}
Let $d$ be a real-valued function on $\stable{G_A}$ and
let $\gamma_v=0$ if $v\in U$ and $\gamma_{\recordnode{T}}=d(T)$ if $T\in \stable{U}$.
Then $[G_A(A),\gamma,0]$ linearizes $d$ if and only of $d$ is nonnegative and
inclusion-wise non-increasing on $\stable{G_A}$.
\end{fact}
\begin{proof}
For each $T\in \stable{U}$, the stable sets in $G(A)$ that meet $A$ in
$T$ are: the set $T$ with $\gamma(T)=0$,
and for each $T'\in \stable{A}$ with $T'\supseteq T$,
the set $T\cup\{\recordnode{T'}\}$ with
$\gamma(T\cup\{\recordnode{T'}\})=d(T')$.
The maximum of these weights is $d(T)$ if and only if
$d(T)\ge 0$ and $d(T)\ge d(T')$ for $T'\subseteq T$, as claimed.
\end{proof}
\noindent
Incidentally, Fact~\ref{record} characterizes the functions that can turn up as
solution $\Service_{\lazy\cU}$ of the servant: each
nonnegative and inclusion-wise non-increasing function on $\stable{\lazy\cU}$.
\begin{fact}\label{servantrecord}
The solution $\Service_{\lazy\cU}$ of the servant is nonnegative and inclusion-wise non-increasing
on $\stable{\lazy\cU}$. Moreover, every nonnegative, inclusion-wise non-increasing
function on the stable sets of a graph arises in this way.
\end{fact}
\begin{proof}
That $\Service_{\lazy\cU}$ is nonnegative and non-increasing is obvious.
For the second statement just take $G$ such that $G_{V_2}$ is a record of $G_{\lazy\cU}$ and
apply Fact~\ref{record}.
\end{proof}
\noindent
Actually, any function on $\stable{G_A}$ is the sum of a constant function, a
linear function and a nonnegative non-increasing function. So by
by toggling in $[G_A(A),\delta,0]$ the values $\delta_v$ with $v\in A$ (which are $0$ in
Fact~\ref{record}) and the third
entry (which is also $0$ in Fact~\ref{record}),
we can find linearizations for any function on $\stable{A}$.
\\[.5\baselineskip]
We say that $H=(W,F)$ {\em linearizes $G_U$ with $\cU$},
if $W\cap V=U$ and $H_U=G_U$, and 
for every nonnegative and inclusion-wise non-increasing function
$\sigma_{\lazy\cU}$ on $\stable{\lazy\cU}$,
there exist $\gamma\in\R^W$, $\sigma\in\R$, such that 
$[H,\gamma,\sigma]$ linearizes $\sigma_{\lazy\cU}\circ\Hom{\cU}$.
\\[.5\baselineskip]
So, instead of taking the term $(\Service_{\lazy\cU}\circ \Hom{\cU})(S_R)$
into the objective function of the master,
we can replace the master-graph by the
{\em linearized master-graph} $G_{V_1\cup U}\cup H$ and
the term $(\Service_{\lazy\cU}\circ \Hom{\cU})(S_R)$ by $\gamma(W\cap S^+)$.
This gives the following alternative for the master.
\\[.5\baselineskip]
\indent
Find in $G_{V_1\cup U}\cup H$, for each stable set $R$ in $Z$,\\
\indent
\indent
a stable set $S^+_R$ with $Z\cap S^+_R=R$, that maximizes $w((V_1\cup U)\cap S^+_R)+ \gamma(W\cap S^+_R)$.\\
\mbox{}
\hfill({\em Linearized master with root})
\\[.5\baselineskip]
The {\em linearized master/servant decomposition} consists of this ``linearized master with root''
together with the original servant with servant-graph $G_{V_2\cup\,\lazy{\cU}}$, root $\lazy\cU$ and
objective function $S_T\mapsto w(V_2\cap S_T)$.
\\[.5\baselineskip]
Note that as defined the linearized master may well be larger then $G$.
If $|W\sm U|\le \tau<|V_2|$, we speak of a
{\em $\tau$-linearized cutset decomposition} and we call
$(V_1,U,V_2)$ with $\cU$ a {\em $\tau$-linearizable cutset separation} and
$U$ a  {\em $\tau$-linearizable cutset}.
We will use $\tau$-linearized decomposition with $\tau\le 1$ and only when the
corresponding linearized master-graph is a proper induced subgraph of $G$.
Mind that the $\tau$-linearized master need not be an induced subgraph of $G$, not even for $\tau=1$.
The servant-graph, as always, is a proper induced subgraph of $G$.
\\[.5\baselineskip]
We use $\tau$-linearized decompositions for algorithms in the same way as decribed for
templates above. When running an algorithm using $\tau$-linearized decomposition,
we will generate a left/right ordered list of rooted graphs as before,
except that now we use linearized masters.
Consequently also the analysis of the running time is almost the same.
Almost! Lemma~\ref{graphcount} does not refer to $\tau$-linearized decompositions
with  $\tau\ge 1$ and we have to account for that.
We only use $0/1$-linearized decompositions.
By Lemma~\ref{quadratic}, their lists have only quadratic length.
\\[.5\baselineskip]
{\bf Templates with singleton-regions.}
Each algorithm in this paper either only use linearized decompositions or
only template decomposition, so without adding extra nodes.
Our template decompositions only come from separations where the grouping consists of singletons.
In such a setting, we denote a region $\cA$ just by its union $A=\cA$;
in line of that, we then denote templates as triples
$(G,Z,\Omega)$ where $\Omega$ consists of subsets of $V$.

If $\Omega$ is a collection of sets in $V$, then $G(\Omega)$ denotes the graph obtained by adding a record $\record{W}$ for each 
$W\in \Omega$. If $\cC$ is a class of triples $(G,Z,\Omega)$ where
$(G,Z)$ s a rooted graph and $\Omega$ a collection of subsets in $V$,
then $\record{\cC}$ denotes all rooted graphs $(G(\Omega),Z)$ with
$(G,Z,\Omega)\in \cC$.
\\[.5\baselineskip]
{\bf Outline.}
In Section~\ref{pfthm1}, we consider node cutsets that 
induce a 3-node path. A 3-node path has 5 stable sets, 
so we get records on 5 nodes then.
Five-node records are already quite big, 
but for the graphs considered in Section~\ref{pfthm1}
they work out fine, see Section~\ref{pfthm1}.
Besides these ``3-node path inducing'' cutsets,
all node cutsets we use are 1-linearizable.
We analyze those in Section~\ref{cliqueroot} and
our usage of node cutsets inducing a 3-node path in Section~\ref{pfthm1}.
\subsection{$\mathbf 0/1$-linearized decompositions}\label{cliqueroot}
Lemma~\ref{graphcount} implies that a decomposition along $0$-linearizable cutsets
$U$, will generate a quadratic of number of graphs.
To understand the structure of these cutsets, observe that, for $[G_U,\gamma,\sigma]$ to linearize a non-increasing function $d$ on $\lazy{\cU}$ forces the values $\sigma=\dd{\lazy{\emptyset}}$ and $\gamma_u=\dd{\lazy{\{u\}}}-\dd{\lazy{\emptyset}}$ for all $u\in U$.
So $[G_U,\gamma,\sigma]$ linearizes $d$
when $U$ is a clique, but not otherwise: if $a$ and $b$ are
nonadjacent nodes in $U$, then the function
$S\mapsto \dd S$ that takes value $1$ if $S=\lazynode\emptyset$ and $0$ otherwise,
has $-1=\dd{\lazy{\{a,b\}}}-\sigma =\gamma_a+\gamma_b=-2$; which is absurd.
If $U$ is a clique and $V_1$ and $V_2$ are both nonempty,
then $U$ is a {\em clique cutset}.
If $U$ is a clique and $V_1=\emptyset$ and $|\cU|<|U|$,
then $U$ contains pair of adjacent twins $u,v$
($u,v$ are {\em twins} if they have the same neighbors in $V\sm\{u,v\})$.
Actually if $u,v$ are adjacent twins, then
$(\emptyset,\{u,v\},V\sm \{u,v\})$ is a $0$-linearizable separation.
All-in-all, clique cutsets and adjacent twins are $0$-linearizable cutsets
and, conversely, each $0$-linearizable cutset is a clique cutset or has an
adjacent twin.
Since clique cutsets can be found in polynomial time (Whitesides~\cite{WHI1})
we can get the following result of Whitesides~\cite{WHI1} from
Lemma~\ref{graphcount}.
We skip the proof as it is a simpler version of what is written in
the proofs of Theorems~\ref{semiamalgamrecursion} and~\ref{amalgamrecursion}.
\begin{corollary}[Whitesides~\cite{WHI1}]\label{cliquerecursion}
Let $\cG$ be a class of graphs closed under clique cutset decomposition.
If $\cP\subseteq \cG$ contains all members of $\cG$
without clique cutsets, 
then the stable set problem on graphs in $\cG$
is solvable in polynomial time
if and only if
the stable set problem on graphs in $\cP$
is solvable in polynomial time.
\end{corollary}
\noindent
So $0$-linearized decompositions of rooted graphs are well-understood:
the node cutsets are cliques,
they can be found in polynomial time,
and the $0$-linearized decomposition-lists have only quadratically
many members
and use only proper induced subgraphs.
The same is true for $1$-linearized decompositions, 
except that it is
not the node cutset but only the servant-root that is guaranteed to be a clique.
\begin{lemma}
\label{1linearizedset}
Let $H$ be a graph on $U\cup\{r\}$ with $r\not\in U$ and let $\cU$ be a
a partition of $U$ in $H$.
Let $P$ be the union of the two element stable sets in $U$.
Then $H$ linearizes $H_U$ with $\cU$ if and only if
$P$ is fully adjacent to $r$ and lies in a set
$A\in \cU\cup\{\emptyset\}$.
\end{lemma}
\begin{proof}
First assume that $H$ linearizes $H_U$ with $\cU$.
Recall,
that $[H,\gamma,\sigma]$ linearizing a nonnegative non-increasing function $d$ on
$\stable{\lazy{\cU}}$ means that
\begin{eqnarray}
\label{1decadj}
\gamma(S)&= &\dd{\lazynode{S}}-\sigma\ \mbox{ if $S\in \stable{U}$ is adjacent to $r$},\\
\max\{0,\gamma_r\}+
\label{1decnonadj}
\gamma(S)&= &\dd{\lazynode{S}}-\sigma\ \mbox{ if $S\in \stable{U}$ is not adjacent to $r$}.
\end{eqnarray}
Applying this to stable sets with at most one element, this forces the values:
\begin{equation}\label{force}
\sigma= \dd{\lazynode{\emptyset}} - \max\{0,\gamma_r\}
\mbox{ and } \gamma_u=
\begin{cases}
\dd{\lazynode{\{u\}}}-\sigma&\mbox{if $u\in U$ is adjacent to $r$},\\
\dd{\lazynode{\{u\}}}-\dd{\lazynode{\emptyset}}&\mbox{if $u\in U$ is not adjacent to $r$}.\\
\end{cases}
\end{equation}
Substituting these back in (\ref{1decadj}) and (\ref{1decnonadj}) gives
for any nonadjacent pair $u, v$ in $U$ the identity:
\begin{equation}\label{force2}
\dd{\lazynode{\{u,v\}}}+
\dd{{\lazynode{\emptyset}}}
-\dd{\lazynode{\{u\}}}-\dd{\lazynode{\{v\}}}
=
\begin{cases}
\max\{0,\gamma_r\}
&\mbox{if $u$ and $v$ are adjacent to $r$},\\
0
&\mbox{if $u$ or $v$ is not adjacent to $r$}.
\end{cases}
\end{equation}%
Consider the nonnegative and non-increasing function $\bar d$ on
$\stable{\lazynode{\cU}}$ given by
$\bardd{\lazynode{\emptyset}}=3, \bardd{\lazynode{\{s\}}}=2$ for all $s\in P$,
and $\bardd{\lazynode{S}}=0$ for all other stable sets in $U$.
Consider any nonadjacent pair $u,v$ in $U$.
The definition of $\bar d$ implies that
$$\bardd{\lazynode{\{u,v\}}}\in\{0,2\}
\mbox{ and }
\bardd{\lazynode{\{u,v\}}}+\bardd{{\lazynode{\emptyset}}}
-\bardd{\lazynode{\{u\}}}-\bardd{\lazynode{\{v\}}}=
\bardd{\lazynode{\{u,v\}}}-1.$$
So, we see from (\ref{force2}) that
$u$ and $v$ are both adjacent to $r$,
that $\bardd{\lazynode{\{u,v\}}}=2$, and that
$\lazynode{\{u,v\}}=\lazynode{\{s\}}$ for some $s\in P$.
However, $\lazynode{\{u,v\}}=\lazynode{\{s\}}$ is
 only possible when
$\lazynode{\{u\}}=\lazynode{\{s\}}=\lazynode{\{v\}}$.
That means that the expression
$\dd{\lazynode{\{u,v\}}}+d_{{\lazynode{\emptyset}}}- \dd{\lazynode{\{u\}}}- \dd{\lazynode{\{v\}}})$
is identical to
$\dd{{\lazynode{\emptyset}}} -\dd{\lazynode{\{u\}}}$---for all $d$, not just for $\bar d$.
Hence, for every nonadjacent $u,v\in U$, condition (\ref{force2}) reads:
\begin{equation}\label{force3}
\max\{0,\gamma_r\}=\dd{\lazynode{\emptyset}}-\dd{\lazynode{\{u\}}}.
\end{equation}%
As this does not depend on $v$, this condition on $\gamma_r$
applies to each $u\in P$.
So the function $u\mapsto \dd{\lazynode{\{u\}}}$ is constant on $P$.
This must hold for each nonnegative non-increasing function $d$
on $\stable{\lazy{\cU}}$. This can only be the case if
$\lazynode{\{u\}}$ is the same set for all $u\in P$. Denote that set by $A$;
it is as claimed in the lemma.
\\[.5\baselineskip]
Now assume $P$ is fully adjacent to $r$ and that
there is a set $A\in \cU\cup\{\emptyset\}$ containing $P$.
Let $d$ be a nonnegative and non-increasing function on $\lazynode{\cU}$.
Define $\sigma$ and $\gamma_u$ with $u\in U$ such that:
$$ \sigma=\dd{\lazynode{A}},\ \ 
\gamma_u=\begin{cases}
	\dd{\lazynode{\emptyset}}- \dd{\lazynode{A}} &\text{if}\ u=r,\\
 	\dd{\lazynode{\{u\}}}-\dd{\lazynode{A}}      &\text{if $u\in U$ is adjacent to}\ $r$,\\
 	\dd{\lazynode{\{u\}}}-\dd{\lazynode{\emptyset}} &\text{if $u\in U$ is not adjacent to}\ $r$.\\
    \end{cases}
$$
Note that $\lazynode{\{u\}}=\lazynode{A}$ if $u\in A$.
So $\gamma_u=0$ if $u\in P$.
Since $\dd{\lazynode{\emptyset}}\ge  \dd{\lazynode{A}}$, it is straightforward
to see that $[H,\gamma,\sigma]$ linearizes $d$.
\end{proof}
\begin{lemma}
\label{1sepdec}
Let $(V_1,U,V_2)$ be a node cutset separation with grouping $\cU$
of rooted graph $(G,Z)$.

If $(V_1,U,V_2)$ is a $1$-linearizable cutset separation with grouping $\cU$,
then there exist three disjoint (possibly empty) sets $A_1,K,A_2$ in $V$ with the following properties:
\begin{enumerate}[(a)]
\itemsep=0cm
\item\label{1sepdec1}
$K$ is a clique, $A_1=U\sm K$ and $A_2\subseteq V_2$.
\item\label{1sepdec2}
$A_1$ is fully adjacent to $K\cup A_2$ and not adjacent to $V_2\sm A_2$.
\end{enumerate}
Conversely, if $A_1,K,A_2$ are disjoint sets in $V$
satisfying (\ref{1sepdec1}) and (\ref{1sepdec2}),
then the following hold.
\begin{enumerate}[$\bullet$]
\item For each $r\in A_2$, the pair 
$(G_{V_1\cup U\cup \{r\}}, Z), (G_{V_2\cup\, \lazy{\cU}}, \lazy{\cU})$ is a
$1$-linearized node cutset decomposition such that 
$G_{V_1\cup U\cup \{r\}}$ and $G_{V_2\cup\,\lazy{\cU}}$ are
proper induced subgraphs of $\cG$ and the servant-root $\lazy{\cU}$
is a clique.
\item If $A_2=\emptyset$, then
 $(G_{V_1\cup U}, Z), (G_{V_2\cup K}, K)$ is a
$0$-linearized node cutset decomposition such that
$G_{V_1\cup U}$ and $G_{V_2\cup K}$ are proper induced subgraphs of $G$ and
the servant-root $K$ is a clique.
\end{enumerate}
\end{lemma}
\begin{proof}
If $(V_1,U,V_2)$ is a $1$-linearizable cutset separation with $\cU$, it
follows from Lemma~\ref{1linearizedset}, that there
exists $A_1\in \cU\cup\{\emptyset\}$, so that
$K=U\sm A_1$ is a clique that is fully adjacent to $A_1$.
Let $A_2=\neighbor{G}{A_1}\cap V_2$.
Since $A_1$ is a group in $G_{V_2\cup U}$, the sets $A_1,K,A_2$ satisfy (a) and (b).

Now, suppose $A_1,K,A_2$ satisfy (a) and (b). If $r\in A_2$, then by Lemma~\ref{1linearizedset}
$H=G_{U\cup \{r\}}$ linearizes $G_U$ with $\cU$. If $A_2=\emptyset$, then $K$ is a clique cutset.
The rest is straightforward.
\end{proof}
\noindent
A graph on $n$ nodes with a clique of size $n-2$ is a {\em near-clique}.
We call a decomposition of a rooted graph $(G,Z)$ {\em proper} if $G$ is not a near-clique.
A decomposition list is {\em proper} if it is obtained by proper decompositions only.
For $G=(V,E)$ and $Z\subseteq V$, we define $\load(G,Z)=|V\sm Z|-1$.
\begin{lemma}\label{quadratic}
A proper $1$-linearized decomposition-list of a rooted graph with $n$ nodes
and a clique as root has most $n^2$ members.
\end{lemma}
\begin{proof}
Consider a proper $1$-linearized decomposition-list $\cG$.
Let $\cG_{> 0}$ consist of the members of $\cG$ with positive load.
We analyze the impact of a single proper $1$-linearized decomposition in $\cG$ on
the following parameters:
\begin{itemize}
\item
The total 
$\Load{\cG}=\sum_{(G',Z')\in \cG_{> 0}}\load(G',Z')$
of the positive loads in $\cG$.
\item
The number $\pos{\cG}$ of members of $\cG$ with positive load.
\item
The total 
$\rootcount{\cG}=\sum_{(G',Z')\in \cG_{> 0}}|Z'|$
of the members of $\cG$ with positive load.
\end{itemize}
Clearly, these numbers satisfy:
\begin{eqnarray}       \label{lowhigh}
\pos{\cG}\le \Load{\cG}.
\end{eqnarray}
Consider a proper
$1$-linearized decomposition $\rootindmaster$, $\rootservant$ of 
a member $(G,Z)\in \cG$ coming from a separation $(V_1,U,V_2)$
of $(G,Z)$ with a grouping $\cU$.
Then $Z$ is a clique and $G$ is not a near-clique.
Recall from the definition of $\tau$-linearized decompositions that
$G_1=G_{V_1\cup U\cup \{r\}}$ for some node $r\in V_2$, if
$U$ is not a clique, and that $G_1=G_{V_1\cup U}$, otherwise.
This gives the following identity:
\begin{eqnarray}       \label{sumup}
\load{\rootindmaster} +\load{\rootservant}=
\begin{cases}
           \load{(G,Z)}-1&\text{ if $U$ is a clique,}\\
           \load{(G,Z)}  &\text{ if $U$ is not a clique}.
\end{cases}
\end{eqnarray}
We distinguish between {\em special} decompositions, when $\load{\rootindmaster}=-1$,
and {\em normal} decompositions, when $\load{\rootindmaster}\ge 0$.
First we analyze special decompositions.
For those $G_1$ has all nodes in $Z$, so $U$ is a clique.
Hence (\ref{sumup}) gives: $\load{\rootservant}=\load{(G,Z)}$, which is positive.
Since $|V_1\cup U|>|\lazy{\cU}|$ and $Z=V_1\cup U$, the
root of $\rootservant$ is smaller than the 
root of $(G,Z)$.
Hence a special decomposition in $\cG$ gives a list $\cH$ with:
\begin{equation}\label{special}
\Load{\cH}=\Load{\cG},\ \ \  \
\pos{\cH}=\pos{\cG},\ \ \  \
\rootcount{\cH}<\rootcount{\cG}.
\end{equation}
Next we analyze normal decompositions in $\cG$, so when $\load{\rootindmaster} \ge 0$.
Since $Z$ is a clique and $G$ is not a near-clique, we have that $\load(G,Z)\ge 2$.
So, by (\ref{sumup}),
at least one of $\rootindmaster$ and $\rootservant$ has positive load.
So a normal decomposition in $\cG$ gives a list $\cH$ with:
\begin{equation}\label{normal}
\Load{\cH}\le \Load{\cG},\ \ \  \
\pos{\cH}\ge \pos{\cG},\ \ \  \
\rootcount{\cH}\le \rootcount{\cG}+n-1.
\end{equation}
Moreover:
\begin{equation}\label{strict}
\Load{\cH}<\Load{\cG}\mbox{ or }
\pos{\cH}>\pos{\cG}.
\end{equation}
Indeed, if $\Load{\cH}=\Load{\cG}$, then
$\load{\rootindmaster} +\load{\rootservant}= \load{(G,Z)}$,
so by (\ref{sumup}), $U$ is not a clique.
That means that $\load\rootservant=|V_2|-1\ge 1$
and that $U$ contains a node that is not in $Z$.
Since the extra node in $G_1$ is not in $Z$ and not in $U$,
we see that also $\load{\rootindmaster}\ge 1$.
Hence, $\pos{\cH}>\pos{\cG}$.
This proves (\ref{strict}).
\\[.5\baselineskip]
Now (\ref{lowhigh})-(\ref{strict}) tell that
when decomposing a rooted graph on $n$ nodes by proper $1$-linearized decompositions
we will make at most $n$
normal decompositions, and that we will, over time, create no more
than $n(n-1)$ root nodes.
So we do at most $n(n-1)$ special decompositions.
So $|\cG|\le n+n(n-1)=n^2$, as claimed.
\end{proof}

\begin{theorem}\label{semiamalgamrecursion}
Let $\cR$ and $\cP$ be classes of rooted graphs
such that there exists a polynomial-time algorithm that for each input
from $\cR\sm \cP$ finds a proper $1$-linearized decomposition
into $\cR$ and such that the stable set problem on rooted graphs in $\cP$
is solvable in polynomial time.  Then the stable set problem on rooted graphs in $\cR$
is solvable in polynomial time.
\end{theorem}
\begin{proof}
It follows from Lemma~\ref{quadratic}, that there is a polynomial time algorithm that
finds for any input $(G,Z)\in \cR$ on $n$ nodes
a proper $1$-linearized decomposition list
$\{(G_1,Z_1), \ldots,$ $ (G_m,Z_m)\}$ in $\cP$ with $m\le n^2$.
With such list at hand, a stable set problem on $(G,Z)$ reduces to solving a stable
set problem on each of the rooted graphs $\{(G_1,Z_1), \ldots, (G_m,Z_m)\}$.
For each $i=1,\ldots,m$ the node weights for the stable set problem on
for $(G_i,Z_i)$ can be determined from the solutions of the problems on
$\{(G_{i+1},Z_{i+1}), \ldots, (G_m,Z_m)\}$.
As each $(G_i,Z_i)$ is in $\cP$, we can solve the full list of problems
in polynomial-time,
provided that we do that going from right-to-left along the list, starting by $(G_m,Z_m)$ and
ending with $(G_1,Z_1)$. Having a found a solution to all problems on the list,
we can construct a solution of the original stable set problem on $(G,Z)$ 
by scanning the list of solutions from left-to-right.
\end{proof}
\subsubsection*{Amalgam decomposition}
An array $(V_1,A_1,K,A_2,V_2)$ of disjoint sets with union $V$ is
an {\em amalgam separation} for $G=(V,E)$, with amalgam $(A_1,K,A_2)$, if 
it has the following properties:
\begin{itemize}
\item
$A_1$ and $A_2$ are nonempty fully adjacent sets.
\item
$K$ is a (possibly empty) clique that is fully adjacent to $A_1\cup A_2$.
\item
$V_1$ is not adjacent to $V_2\cup A_2$ and
$V_2$ is not adjacent to $V_1\cup A_1$.
\item
$|V_1\cup A_1|\ge 2$ and $|V_2\cup A_2|\ge 2$.
\end{itemize}
Note that it is allowed that $K$ has neighbors in $V_1\cup V_2$.
\\[.5\baselineskip]
Amalgams were introduced by Burlet and Fonlupt \cite{BF1}
to design a polynomial-time algorithm to recognize Meyniel graphs,
a special class of perfect graphs.
Cunningham and Cornu\'ejols \cite{CC} designed a
polynomial-time algorithm that finds an amalgam separation
or decides that none exists.
In \cite{BF1,CC}, a graph with an amalgam separation
$(V_1,A_1,K,A_2,V_2)$ is decomposed into two {\em amalgam blocks} $G_1$ and $G_2$,
where, for both $i=1$ and $i=2$, the graph $G_i$ is obtained from $G_{V_i\cup A_i\cup K}$ by adding a single
new node that is fully adjacent to $A_i\cup K$.
We call the pair $G_1,G_2$ an {\em amalgam decomposition} of $G$
corresponding to $(A_1,K,A_2)$.
\\[.5\baselineskip]
Amalgams give rise to $1$-linearizable cutsets and
amalgam decompositions are $1$-linearized decompositions;
this is the next lemma.
\begin{lemma}\label{amalgamisonelin}
Let $(A_1,K,A_2)$ be an amalgam of a graph $G$.
Then $A_1\cup K$ is a $1$-linearizable cutset of $(G,\emptyset)$
with grouping $\cU=\{A_1\}\cup\{\{u\}:u\in K\}$.
Moreover, $(G,\emptyset)$ has a $1$-linearized decomposition such that
the master-graph and the server-graph form an amalgam decomposition of 
$G$ corresponding to amalgam $(A_1,K,A_2)$,
the master-root is empty, and the servant-root is the clique
$\lazy{\cU}$.
 \end{lemma}
\begin{proof}
Let $(V_1,A_1,K,A_2,V_2)$ be an amalgam separation.
Then $(V_1,A_1\cup K,V_2\cup A_2)$ is a node cutset separation of $(G,\emptyset)$
with grouping $\cU=\{A_1\}\cup\{\{u\}:u\in K\}$.
Observe that $A_1,K,A_2$ satisfy (\ref{1sepdec1}) and (\ref{1sepdec2}) of
Lemma~\ref{1sepdec} with respect to that node cutset separation.
So $A_1\cup K$ is $1$-linearizable.
By the definition of amalgam, $A_2\neq\emptyset$. Take $r\in A_2$.
Now the second part of Lemma~\ref{1sepdec} tells  that the pair
$(G_{V_1\cup A_1\cup K\cup\{r\}},\emptyset)$,
$(G_{V_2\cup A_2\cup \, \lazy{\cU}}, \lazy{\cU})$ is
a $1$-linearized cutset decomposition.
It is obvious that 
$G_{V_1\cup A_1\cup K\cup\{r\}}$ and
$G_{V_2\cup A_2\cup\,\lazy{\cU}}=
G_{V_2\cup A_2\cup K\cup\{\lazynode{A_1}\}}$
 is an amalgam decomposition
corresponding to amalgam decomposition
$(V_1,A_1,K,A_2,V_2)$.
Clearly, $\lazy{\cU}$ is a clique.
\end{proof}%
\begin{corollary}
If $\cG$ is an amalgam decomposition-list of a graph $G=(V,E)$,
then $|\cG|\le |V|^2$.
\end{corollary}
\begin{proof}
By Lemma~\ref{amalgamisonelin}, there is a map $H\mapsto Z_H$ from
$\cG$ to cliques, so that
$\{(H,Z_H)\,:\,H\in \cG\}$ is a $1$-linearized decomposition-list of
$(G,\emptyset)$.
Now Lemma~\ref{quadratic} yields $|\cG|\le |V|^2$.
\end{proof}
\noindent
Theorem~\ref{capdec} uses amalgams to describe the structure
of cap-free graphs with no even holes.
When we wanted to use that to design
an algorithm for the stable set problem on
these graphs, we ran into the multiple-problem aspect of the servant.
Cornu\'ejols and Cunningham \cite{CC} can find an amalgam separation with
minimal servant, but as illustrated by Figure \ref{hiddenamalgam},
that does not guarantee that the
amalgam blocks will have no amalgams. 
This lead us to the rooted graph approach of this paper.
It is essential here and to our knowledge new;
till today we know of no other way to use
amalgam decompositions in polynomial-time algorithms for 
finding maximum weight stable sets.
\begin{theorem}\label{amalgamrecursion}
Let $\cG$ be a class of graphs closed under amalgam decomposition.
If $\cP\subseteq \cG$ contains all near-cliques
and all members of $\cG$ without amalgams,
then the stable set problem on graphs in $\cG$ 
is solvable in polynomial time
if and only if the stable set problem on graphs in $\cP$ 
is solvable in polynomial time.
\end{theorem}
\begin{proof}
Let $\cR$ be the class of rooted graphs $(G,Z)$ such that $G\in \cG$ and 
$Z$ is a clique. The stable set problem on a graph $G\in \cG$
is a rooted stable set problem on $(G,\emptyset)$, which is in $\cR$.
So it suffices to prove that the stable set problem on rooted graphs
in $\cR$ is solvable in polynomial time.
Let $\cQ$ the class of rooted graphs $(G,Z)$ with $G\in \cP$ and
such that $Z$ is a clique. 
As a clique $Z$ has only $|Z|+1$ stable sets,
the stable set problem on rooted graphs in $\cQ$ is solvable in polynomial time.
So by Theorem~\ref{semiamalgamrecursion}
it suffices to design
a polynomial-time algorithm that finds a proper $1$-linearized decomposition for
any rooted graph $(G,Z)$ in $\cR\setminus Q$.

Here is this algorithm:
If $G$ is a near-clique, $(G,Z)$ is in $\cQ$.
Otherwise, use the algorithm of Cornu\'ejols and Cunningham \cite{CC}
to search for amalgams in $G$. If none is found:
$G\in \cP$, so in $(G,Z)\in \cQ$.
If an  amalgam separation $(V_1,A_1,K,A_2,V_2)$ is found, proceed as follows
to find a $1$-linearized decomposition for $(G,Z)$, which will be proper as $G$ is not a near-clique.
If the clique $Z$ meets $V_1\cup V_2$, the root $Z$ is contained in
$V_1\cup A_1\cup K$ or in $V_2\cup A_2\cup K$, so then
one of the node cutsets $A_1\cup K$ and $A_2\cup K$
yields a $1$-linearized decomposition of $(G,Z)$.
If $Z$ does not meet $V_1\cup V_2$, it lies in 
$A_1\cup K\cup A_2$ and thus $K\cup Z$ is a clique.
Since $G$ is not a near-clique, it has at least 3 nodes outside $K\cup Z$.
Assume two of those lie in $A_2\cup V_2$.
Then the node cutset $A_1\cup K\cup Z$
yields a $1$-linearized decomposition of  $(G,Z)$.
\end{proof}
\begin{figure}[htb]
    \hspace{2.5cm}
    \includegraphics[height=5cm]{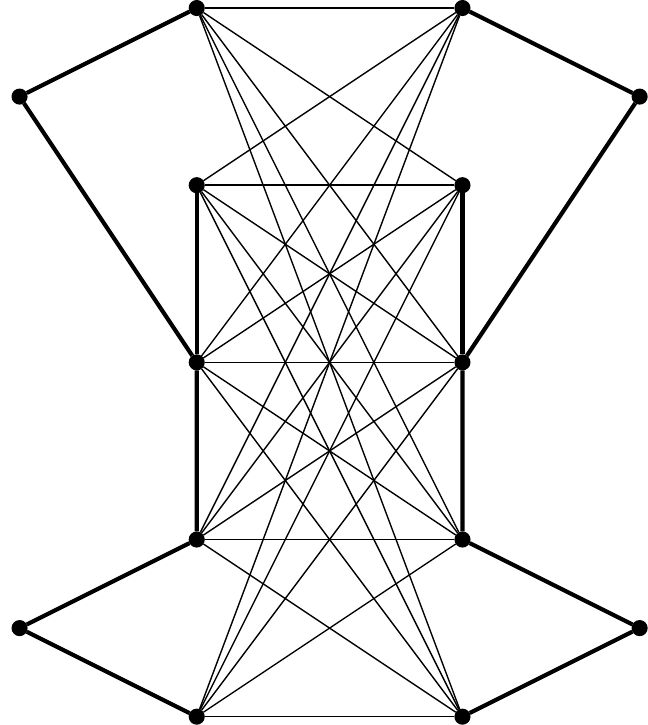}
    \hfill
    \includegraphics[height=5cm]{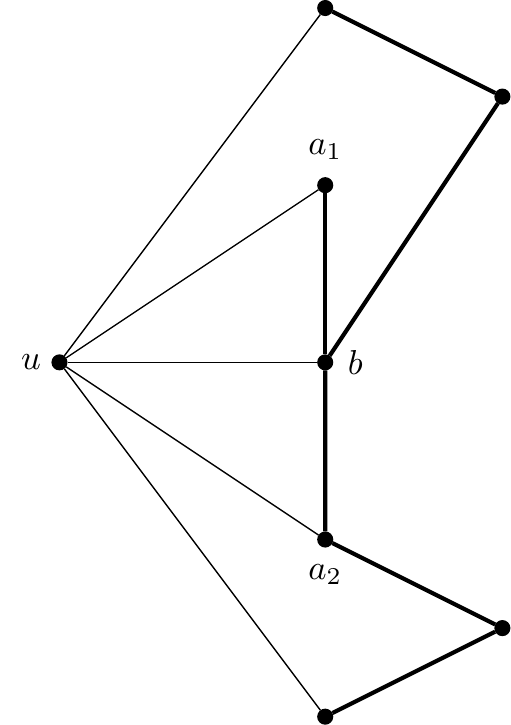}
    \hspace{2.5cm}
    \caption{The amalgam of the graph on the left is unique
(in this case the clique $K$ is empty). Its blocks have amalgams,
for instance $(\{a_1,a_2\},\{u\},\{b\})$.\label{hiddenamalgam}}
\end{figure}%
\subsection{Decomposing into fan-templates---Proof of Theorem 1}\label{pfthm1}
To prove Theorem \ref{algtheorem} we use Theorem \ref{capdec}, which
decomposes cap-free graphs into triangle-free graphs, and
Theorem \ref{tf}, which further decomposes triangle-free ``odd-signable''
graphs into ``the cube'' and ``fan-templates''.

A graph is {\em odd-signable} if it contains a set $F$ of edges such that
$|F\cap C|$ is odd for each chordless cycle $C$.
Triangle-free graphs with no even holes are clearly odd-signable.

{\em The cube} is the unique 3-regular bipartite graph on 8 vertices, so that is
the 1-skeleton of the three dimensional cube. The cube is odd-signable.

A {\em fan with base $(u,c,v)$} consists of an
$uv$-path $P$ together with a node
$c$ adjacent to a subset of nodes of $P$ including $u$ and $v$.
If $Z$ is a subset of the base of a fan $G=(V,E)$ and
$\Omega$ is a collection of triples in $V$, then we call
$(G, Z, \Omega)$ a {\em fan-template}.
If the triples in $\Omega$ each induce a subpath of one of the holes of $G$,
we call the fan-template {\em good}.
The following results say that good fans do come up in decomposing cap-free
odd-signable graphs and that they are well tractable.
\begin{theorem}[{\cite[Theorems 2.4 and 6.4]{TFODD}}]\label{tf}
If $G=(V,E)$ is a triangle-free odd-signable graph that is not isomorphic to the cube
and has no clique cutset,
then the template $(G,\varnothing,\varnothing)$ can in polynomial time
be decomposed into list of at most $|V|$ good fan-templates.
\end{theorem}%
\begin{proof}
By \cite[Theorem 2.4]{TFODD},
$G$ has no induced subgraph isomorphic to the cube.
Now \cite[Theorem 6.4]{TFODD} says that $G$ can be obtained from the hole
by a sequence of ``good ear additions'' (defined in \cite[Definition 6.1]{TFODD}).
An ear addition is the reverse of a node cutset decomposition where
the servant graph is a fan and the node cutset is the base of the fan.
So reversing the sequence of good ear additions amounts to a decomposition of
$G$ into fan-templates. The regions of these templates are the locations where the
ears are added and the goodness of these ear additions means
that the fan-templates are good.
Since adding an ear increases the size of the graph,
we obtain at most at most $|V|$ good fan-templates.
As the node cutsets needed for this decomposition are triples, we can find them in polynomial time, by enumeration.
\end{proof}
\begin{lemma}\label{zipfan}
If $(G, Z, \Omega)$ is a good fan-template with base $(u,c,v)$ and $n$ nodes,
then $G(\Omega)-c$ can be decomposed
along 2-node clique cutsets into a list of at most $|\Omega|+n$ graphs, each with at most 8 nodes.
\end{lemma}
\begin{proof} 
The graph $G(\Omega)-c$ consists of the path $G-c$
together with all the records.
As each region of a fan-template is a 3-node subpath of
one of the holes of $G$, each record has 5 nodes and is attached in
$G(\Omega)-c$ to a 2- or 3-node subpath of $G(\Omega)-c$.
Note that each edge of the path $G(\Omega)-c$ forms a 2-node clique cutset
of $G(\Omega)-c$ (except maybe the first or the last edge of $G(\Omega)-c$).
If we decompose $G(\Omega)-c$ along all these 2-node clique
cutsets, we obtain collection of graphs, each consisting of
a 2- or 3-node subpath of $G(\Omega)-c$ together with at most one of the records.
Such graphs have at most 8 nodes.
\end{proof}
\begin{lemma}\label{lem2gen}
The stable set problem on rooted record graphs of good fan-templates
is solvable in polynomial time.
\end{lemma}
\begin{proof}
Let $(u,c,v)$ be the base of a good fan-template $(G, Z, \Omega)$.
By Corollary \ref{cliquerecursion} and Lemma \ref{zipfan},
for each of the (at most 5) stable sets $T$ in $Z$,
we can find in polynomial time,
a stable set $S_T$ in 
$G(\Omega)$ that has maximum weight
among those that intersect $Z$ in $T$.
Among these stable sets $S_T$, we choose the best one.
\end{proof}
\begin{theorem}\label{thm2gen}
The stable set problem on cap-free odd-signable graphs
is solvable in polynomial time.
\end{theorem}
\begin{proof}
If node $u$ in graph $G$ is adjacent to all other nodes, then we can set it aside
to compare it with a maximum weight stable set in $G-u$, once we found that.
The stable set problem on the cube can be found by
enumeration.
So the result follows from 
Theorems \ref{capdec},
\ref{amalgamrecursion},
\ref{th:templ-stablesets},
\ref{tf},
Corollary \ref{cliquerecursion},
 and Lemma \ref{lem2gen}.
\end{proof}%
\section{Stable set polytopes}
We examine extended formulations for the stable set polytope of a graph
that admits certain decompositions into smaller graphs and combine
formulations for these smaller parts to one for the whole graph.
We apply this to cap-free odd-signable graphs and thus prove Theorem~2.
\\[.5\baselineskip]
{\bf Notation.}
We denote the convex hull of characteristic vectors of stable sets in
a graph $G$ by $\pol{G}$.
If $\cL$ is a collection of cliques in $G$,
we denote the collection of stable sets in $G$ that intersect
each member of $\cL$ by $\stable{G, \cL}$ and the convex hull
of the characteristic vectors of these stable sets by $\pol{G,\cL}$.
So $\pol{G}=\pol{G,\emptyset}$
and $\pol{G,\cL}$ is the face of $\pol{G}$
obtained by setting at equality all clique constraints
associated to the cliques in $\cL$.
If $x\in \R^V$ and $H=G_U$, we denote the restriction of $x$ to
$U$ by $x_H$, of by $x_U$. 
\subsection{Extended formulations and records}
Extra variables used in extended formulations here mostly come
from records.

Consider a set of nodes $U$ in a graph $G$.
Recall from Section~\ref{decomposition}, that $G(U)$ denotes the graph obtained
by adding to $G$ a clique $\record{U}$ (the record) consisting of new nodes
$\recordnode{S}$, one for each $S \in \stable{G_U}$,
and connect each such $\recordnode{S}$
to all nodes in $U\sm S$.
We denote by 
$\cL(U)$ the collection consisting of the clique 
$\record{U}$ together with all the cliques
$\{v\} \cup \{\recordnode{T}:v\not\in T\in \stable{G_U}\}$ with $v\in U$.

The following result says that any (extended) formulation for
$\pol{G(U), {\cL}(U)}$ is an extended formulation for $\pol{G}$.
\begin{lemma}\label{lift}
Let $U$ be a set of nodes in a graph $G$. Then each stable set $S$ in $G$ has a
unique extension to a member of $\stable{G(U),{\cL}(U)}$,
namely $S\cup\{\recordnode{S\cap U}\}$, and
$$\pol{G}= \{x_G:x\in \pol{G(U),{\cL}(U)}\}.$$
If, moreover, $\cL$ is a collection of cliques in $G$, then
$$\pol{G,{\cL}}
= \{x_G:x\in \pol{G(U), \cL\cup{\cL}(U)}\}.$$
\end{lemma}
\begin{proof}
Proving that $S \cup \{\recordnode{S\cap U}\}$ is the only extension of $S$ in
$\stable{G(U),{\cL}(U)}$ is straightforward.
Moreover, if $S$ meets $\cL$, then so does $S \cup \{\recordnode{S\cap T}\}$.
The rest now follows as each of $\pol{G},\pol{G(U),{\cL}(U)}$,
$\pol{G,\cL}$, and $\pol{G(U), \cL\cup{\cL}(U)}$ are
convex hulls of stable sets.
\end{proof}
\subsection{Composing across node cutsets}\label{nodecutsystems}
Lemma~\ref{lift} expresses the stable set polytope of a graph $G$ 
as a particular face of the stable set polytope of the record graph of $U$ and $G$.
Our next result is that those particular faces
admit a simple composition rule when $U$ is a node cutset.
\begin{theorem}\label{composesystem}
Let $(V_1,U,V_2)$ be a node cutset separation of graph $G=(V,E)$.
Moreover, let $\cL_1$ be a collection of cliques in $G_1=G_{V_1\cup U}$ and
let $\cL_2$ be a collection of cliques in $G_2=G_{V_2\cup U}$. 
Then, each $x\in\R^{V\cup \record{U}}$ satisfies:
\begin{multline*}
x \in \pol{G(U),{\cL}_1\cup{\cL}_2\cup{\cL}(U)}
\mbox{ if and only if }\\
x_{G_1(U)}\in \pol{G_1(U),{\cL}_1\cup{\cL}(U)}
\mbox{ and }
x_{G_2(U)}\in \pol{G_2(U),{\cL}_2\cup{\cL}(U)}.
\end{multline*}
Hence
$$\pol{G} = \{x\in \R^V:\exists_{y\in\R^{\record{U}}}\,
[\,(x_{G_1},y)\in \pol{G_1(U),{\cL}(U)}
\mbox{ and }
(x_{G_2},y)\in \pol{G_2(U),{\cL}(U)}\,]\}.$$
\end{theorem}
\begin{proof}
By Lemma~\ref{lift}, the second assertion follows from the first one.
The ``only if'' direction of the first assertion is obvious.
For the ``if'' direction, it suffices to consider 
$x\in \Q^{V\cup\record{U}}$.

Assume
$x_{G_1(U)}\in \pol{G_1(U),{\cL}_1\cup{\cL}(U)}$
and
$x_{G_2(U)}\in \pol{G_2(U),{\cL}_2\cup{\cL}(U)}$.
Then, for $i=1,2$, there exists a positive integer $n_i$,
so that $n_ix_{G_i}(U)$
is the sum of the characteristic vectors of a collection of
(not necessarily distinct) stable sets $S^i_1,\ldots,S^i_{n_i}$ in $G_i(U)$.
By replicating members in these two collections of stable sets (if necessary),
we may assume that $n_1=n_2$; let $n=n_1=n_2$.

Since $x_{G_i(U)}\in \pol{G_i(U),{\cL}(U)}$,
each 
$S^i_1,\ldots,S^i_n$ meets each clique in ${\cL}(U)$ exactly once.
So, for each stable set $S$ in $U$,
the number of sets among $S^i_1,\ldots,S^i_n$ that 
intersect $U$ in $S$ is equal to
$n(x_{G_i(U)})_{{\recordnode{S}}}=
nx_{{\recordnode{S}}}$.
As this applies to both $i=1,2$, we can renumber $S^2_1,\ldots,S^2_n$ so that
$S^1_j\cap U=S^2_j\cap U$ for $j=1,\ldots,n$.
Doing so,
 $x$ is a convex combination of the characteristic vectors
of the stable sets $S^1_1\cup S^2_1, \ldots, S^1_n\cup S^2_n$.
Since
$S^1_j\in \pol{G_1(U),{\cL}_1\cup {\cL}(U)}$
and
$S^2_j\in \pol{G_2(U),{\cL}_2\cup {\cL}(U)}$
for all $j$, each $S^1_j\cup S^2_j$ is in $\stable{G(U),{\cL}_1\cup{\cL}_2\cup{\cL}(U)}$.  Hence, $x\in \pol{G(U),{\cL}_1\cup{\cL}_2\cup{\cL}(U)}$, as claimed.
\end{proof}
\noindent
Consider Theorem~\ref{composesystem} in case $U$ is a clique cutset.
Then
$\record{U}=\{\recordnode{\varnothing}\}\cup \{\recordnode{\{v\}}:v\in U\}$.
Moreover, for $(x,y)\in \R^V\times \R^{\record{U}}$, we have that
$(x,y)\in \pol{G(U),\cL\cup{\cL}(U)}$
if and only if:
\begin{equation}\label{eq:clique-transf}
x_G\in \pol{G,\cL},\ \ 
 y_{{\recordnode{\varnothing}}}=1-\sum_{v\in U}x_v,
\mbox{\ \  and\ \  }
y_{{\recordnode{\{v\}}}}= x_v\, (v\in U).
\end{equation}
Applying this to each of the three graphs $G, G_1, G_2$ in
Theorem~\ref{composesystem}, we obtain the following result of Chv\'atal.
\begin{corollary}[Chv\'atal \cite{CHV}]\label{cliquecut}
Let $(V_1,U,V_2)$ be a clique cutset separation of a graph $G=(V,E)$ and
let $G_1=G_{V_1\cup U}$ and $G_2=G_{V_2\cup U}$.
Then:
$$\pol{G} = \{x\in \R^V:
x_{G_1}\in \pol{G_1} \mbox{ and } x_{G_2}\in \pol{G_2}\}.$$
If, moreover, ${\cL}_1$ is a collection of cliques in $G_1$
and ${\cL}_2$ is a collection of cliques in $G_2$,
then each $x\in \R^V$ satisfies:
$$x\in \pol{G,{\cL}_1\cup {\cL}_2}
\mbox{ if and only if }
x_{G_1}\in \pol{G_1,{\cL}_1}
\mbox{ and }
x_{G_2}\in \pol{G_2,{\cL}_2}.$$
\end{corollary}
\noindent
In Corollary~\ref{cliquecut}, we can not drop the condition that $U$ is a clique.
Indeed, let $u$ and $v$ be two nonadjacent nodes in $U$
and suppose $G_1$ has a chordless even $uv$-path $Q_1$
and $G_2$ has a chordless odd $uv$-path $Q_2$.
Consider the vector $x\in\R^V$ with $x_v=1/2$ if $v$ lies on $Q_1\cup Q_2$ and $x_v=0$ otherwise.
Then $x\not\in \pol{G}$, but $x_{G_1}\in \pol{G_1}$ and $x_{G_2}\in \pol{G_2}$.
\\[.5\baselineskip]
Balas~\cite{B} has shown how to obtain an extended formulation for the convex hull
of polytopes $P_1$, \ldots, $P_k$, whose size is approximately the sum of the sizes
of the descriptions for these polytopes.
If $A^ix+B^iy\le d^i$ is an extended formulation for $P_i (i=1,\ldots,k)$,
then Balas's formulation for the convex hull reads:
\begin{eqnarray}\label{balas}
x=x^1+\cdots +x^k,\ \lambda_1+\cdots+\lambda_k=1;\ \ 
A^ix^i+B^iy^i-\lambda_id^i\le 0,\ \lambda_i\ge 0\ (i=1,\ldots,k).
\end{eqnarray}
This formula can be used to construct a formulation for $\pol{G}$
from such formulations for parts of a node cutset decomposition of $G$.
Let $H$ be one of these parts and let $U$ denote the node cutset.
For every stable set $S$ in $U$, a
description of the face of $\pol{H(U)}$ given by $x_{\recordnode {S}}=1$
can be inferred from any linear description of
the face $\{x\in \pol{H}:x_v=1\, (v\in S)\}$ of $\pol{H}$.
Since $\pol{H(U), \cL(U)}$ is the convex hull of these faces, Balas's formula~(\ref{balas})
gives an extended formulation for $\pol{H(U), \cL(U)}$
whose size is in the order of $|\record U|=|\stable{H_U}|$ times the size of the linear description of $\pol{H}$.
If we apply this to each part $H$ of the decomposition and
combine the resulting formulations into one list of linear inequalities, we obtain,
by Theorem~\ref{composesystem}, an extended formulation for $\pol{G}$.
This leads to the following result.
\begin{theorem}\label{th:templ-extform}
Let $G$ be a graph and
$\{(G_1,Z_1,\Omega_1), \ldots, (G_k,$\ $Z_k,\Omega_k)\}$
be a decomposition-list of $(G,\varnothing,\varnothing)$.
Assume we are given for each $i=1,\ldots,k$ an extended formulation with size $m_i$
for $\pol{G_i({\{Z_i\}\cup\, \Omega_i})}$.
Then there exists an extended formulation for $\pol{G}$ with size at most $O(k)+m_1+\cdots+m_k$.
\end{theorem}
\begin{proof}
Recursively apply the following immediate corollary of Theorem~\ref{composesystem}:
if $(V_1,U,V_2)$ is a cutset separation of template $(G,Z,\Omega)$ with
master template $(G_1,Z,\Omega_1)$ and servant template $(G_2,U,\Omega_2)$,
then a vector $x$ lies in
$\pol{G(\{Z,U\}\cup\, \Omega),{\cL}(U))}$
if and only if
$x_{G_1(\{Z, U\}\cup \Omega_1)}
\in \pol{G_1(\{Z, U\}\cup \Omega_1),
{\cL}(U)}$
and 
$x_{G_2(\{U\}\cup \Omega_2)}
\in \pol{G_2(\{U\}\cup \Omega_2),
{\cL}(U)}$.
\end{proof}
\noindent
An alternative for adding a record to a graph $G$ is
{\em lifting a node set $U$ to a clique}.
This amounts to deleting $U$ from $G$ and replacing it by a clique
with node set $\lift U$,
and connecting each $\liftnode S\in \lift U$ with
each node in $\neighbor{G}{S}\sm U$.
We call the new graph the {\em clique lift} of $U$ from $G$.
An advantage of clique lifts over records is that clique lifts
yield extended formulations for stable sets that do not
involve ``tight clique constraints'':
$x(K)=1\, (K\in \cL(U))$.
\begin{lemma}\label{cliquelift}
Let $G^+$ be the clique lift of $U\subseteq V$ from a graph $G=(V,E)$.
Then the stable set polytope $\pol{G}$
is the image of $\pol{G^+}$ under the projection
$p:\R^{(V\sm U)\cup(\lift{U})}\rightarrow \R^V$ defined by
\begin{displaymath}
  p_v(x) =\begin{cases}
     \sum_{S\in \stable{G_U},  S\ni v} x_{\powernode S}&\text{if  } v\in U\\
      x_v &\text{otherwise}\,.
    \end{cases}
\end{displaymath}
\end{lemma}
Lifting a node cutset to a clique turns it into a clique cutset.
So we get the following consequence of Corollary~\ref{cliquecut}.
\begin{corollary}
Let $(V_1,U,V_2)$ be a node cutset separation of graph $G=(V,E)$. Moreover, let
$G^+, G^+_1,$ and $G^+_2$ be the clique lifts of $U$
from $G, G_{V_1\cup U}$, respectively $G_{V_2\cup U}$.
Then each $x\in \R^{(V\sm U)\cup(\lift U)}$ satisfies:
$$ x \in \pol{G^+} \mbox{ if and only if }
x_{G^+_1}\in \pol{G^+_1}
\mbox{ and }
x_{G^+_2}\in \pol{G^+_2}.$$
\end{corollary}
\subsection{Generalized amalgams}
We give a decomposition rule for stable set polytopes of graphs 
that admit a {\em generalized amalgam separation};
for a graph with node set $V$, this
is a pair $(U, \cW)$ with $U\subseteq V$
such that $\cW$ is a partition of $V\setminus U$
into nonempty sets $W$ that each have the property that
each node in $W$ with a neighbor outside $W\cup U$ is fully adjacent to $U$.

Generalized amalgam separation unifies a great variety of known
separations. Clique cutset separation and amalgam separation are obvious
special cases.
A notable other example is the ``strip-structure for trigraphs''
introduced by Chudnovsky and Seymour~\cite{CS}.
Faenza, Oriolo, and Stauffer~\cite{FOG} used strip-structures to obtain
extended formulations and polynomial-time algorithms for stable sets problems
in ``claw-free'' graphs. The ``2-clique-bonds'' that
Galluccio, Gentile, and Ventura~\cite{GGV} use to compose linear formulations 
of stable set problems are generalized amalgam separations as well.
\\[.5\baselineskip]
Before actually decomposing a graph along a generalized amalgam
separation $(U, \cW)$ we first lift $U$ to a clique,
$K$ (say). 
Then $(K, \cW)$ is a generalized amalgam separation of the clique lift
and
all structure of $(U, \cW)$ and the original graph
fully carries over to $(K, \cW)$ and the clique lift, 
except for the internal structure of $U$ resp. $K$.
Since Lemma~\ref{cliquelift} explains the effect of clique lifts to
the stable set polytope,
it is enough to investigate $(K,\cW)$ in the
clique lift;  we call the clique lift $G$.
\\[.5\baselineskip]
So $(K, \cW)$ is a generalized amalgam separation of a graph $G$ and $K$ is
a clique.

For $W\in \cW$, we denote by $\gaA{W}$
the collection of equivalence classes in $\buur{G-K}{W}$ of the relation
``having the same neighbors outside $W$''. Related to $\gaA{W}$
we will consider a clique $\power{\gaA{W}}$ consisting of
new nodes $\recordnode{\cX}$,
one for each subcollection $\cX$ of ${\gaA{W}}$.

The {\em generalized amalgam decomposition of $G$ along $(K, \cW)$}
consists of a collection of graphs $G(K,W)$, one for each $W\in \cW$,
together with a ``connecting'' graph $\connect G(K,\cW)$.
Each graph $G(K,W)$ is obtained from the disjoint union of
$G_{K\cup W}$ and $\power{\gaA{W}}$ by connecting each
$\recordnode{\cX}\in \power{\gaA{W}}$ to all nodes in 
$\buur{G-K}{W}\sm \bigcup \cX$ and to all nodes in $K$.
The graph $\connect G(K,\cW)$ is obtained from the disjoint union of
the clique $K$ and all cliques $\power{\gaA{W}}$ with $W\in \cW$,
by adding edges from each node in $K$ to all nodes in
all cliques $\power{\gaA{W}}$ and by adding
all edges $\recordnode{\cX}\recordnode{\cX'}$ 
such that
$\cup \cX$ and $\cup \cX'$ are adjacent in $G$ and
$\cX\subseteq \gaA{W}, \cX'\subseteq \cA_{W'}$,
$W,W'\in \cW, W\neq W'$.
We also define $\cL(K,\cW)=
\{K\cup \power{\gaA{W}}:W\in \cW\}$.
\begin{theorem}\label{graph_composition_extendended_formulation}
Let $(K,\cW)$ be a generalized amalgam separation of a graph $G=(V,E)$ such
that $K$ is a clique.
Moreover, let $\cL= \cL(K,\cW)$.
Then $\pol{G}$ consists of the restrictions $x_G$ of those 
$x\in \R^{V\cup (\bigcup \{\power{\gaA{W}}:W\in \cW\})}$ with
\begin{equation}\label{extamalform}
x_{\connect{G}(K, {\cW})}\in \pol{\connect{G}(K, \cW),\cL}
\mbox{ and }
x_{G(K,W)}\in \pol{G(K,W)}
\mbox{ for all } W\in \cW.
\end{equation}
\end{theorem}
\begin{proof}
Let $H$ be the graph
defined as follows:
the node set of $H$ is $V\cup (\bigcup \{\power{\gaA{W}}:W\in \cW\})$
and the edge set of $H$ is
the union of the edge set of $\connect{G}(K, \cW)$
with the edge sets of all graphs $G(K,W)$ with $W\in \cW$. 

Since each member of $\cL$ is a clique cutset of $H$, it
follows from Corollary~\ref{cliquecut} that $x\in \pol{H,\cL}$ if and only if
$x$ satisfies (\ref{extamalform}).
Hence, it suffices to prove that $\pol{G}=\{x_G:x\in \pol{H,\cL}\}$.
For that it suffices to prove that the function
$S\mapsto S_G$ maps $\stable{H,\cL}$ onto
$\stable{G}$.

First consider $S\in \stable{H,\cL}$. We prove that $S_G\in \stable{G}$.
If $S\cap K\neq \emptyset$, then 
$S\subseteq V\sm \bigcup\{\buur{G}{W}:W\in \cW\}$,
so $S\in \stable{G}$.
Hence we may assume that $S\cap K = \emptyset$.
Then there exists, for each $W \in \cW$,
a collection $\cX_W\subseteq \gaA{W}$ with 
$S\cap \power{\gaA{W}}=\{\recordnode{\cX_W}\}$.
Since $S$ is a stable set in $H$, we have that
$S\cap \buur{G-K}{W} \subseteq \buur{G-K}{W}\sm \neighbor{H}{\recordnode{\cX_W}}=
\bigcup \cX_W$.
Now consider $W, W'\in\cW$ with $W'\neq W$.
Then in $H$, node
$\recordnode{\cX_W}$ is not adjacent to node $\recordnode{\cX_{W'}}$.
Hence $\bigcup{\cX_W}$ and $\bigcup{\cX_{W'}}$ are not adjacent in
$G$.
From this it follows that $S_G$ is a stable set in $G$, as claimed.

Next consider $S'\in \stable{G}$.
We prove that there exists an $S\in \stable{H,\cL}$ with $S'=S_G$.
If $S'\cap K\neq \emptyset$, we just take $S=S'$.
Indeed, in that case,
$S'\subseteq V\sm  \bigcup\{\buur{G}{W}:W\in \cW\}$,
so $S' \in \stable{H, \cL}$.
Hence, we may assume $S'\cap K=\emptyset$.
For each $W\in \cW$,
let $\cX_W$ be the members of $\gaA{W}$ that contain an element of $S'$.
Define $S=S'\cup\{\recordnode{\cX_W}:W\in \cW\}$.
Then $S\in \stable{H,\cL}$ and $S_G=S'$, as required.
\end{proof}
\subsubsection*{Amalgams}
If graph $G=(V,E)$ has an amalgam separation $(V_1,A_1,K,A_2,V_2)$, then
$(K,\{V_1\cup A_1,V_2\cup A_2\})$
is a generalized amalgam separation and $K$ is a (possibly empty) clique.
By Theorem~\ref{graph_composition_extendended_formulation},
$\pol{G}$ consists of the restrictions $x_G$ of all vectors
$x\in\R_+^{V \cup\{\recordnode{\{A_1\}},\recordnode{\varnothing_1},
\recordnode{\{A_2\}},\recordnode{\varnothing_2}\}\}}$ with
\begin{eqnarray}
x_{V_1\cup A_1\cup K\cup\{\recordnode{\{A_1\}},\recordnode{\varnothing_1}\}}
    &\in& \pol{G(K,V_1\cup A_1)},\label{left}\\
x_{V_2\cup A_2\cup K\cup\{\recordnode{\{A_2\}},\recordnode{\varnothing_2}\}}
    &\in& \pol{G(K,V_2\cup A_2)},\label{mid}\\
x_{K \cup\{\recordnode{\{A_1\}},\recordnode{\varnothing_1}, \recordnode{\{A_2\}},\recordnode{\varnothing_2}\}}
    &\in& \pol{\connect{G}(K,\{V_1\cup A_1,V_2\cup A_2\}), \cL},\label{right}
\end{eqnarray}
where $\cL$ consists of the two cliques
$K \cup \{\recordnode{\{A_1\}},\recordnode{\varnothing_1}\}$ and
$K \cup \{\recordnode{\{A_2\}},\recordnode{\varnothing_2}\}$.

For $x\in\R_+^{K \cup\{\recordnode{\{A_1\}},\recordnode{\varnothing_1},
\recordnode{\{A_2\}},\recordnode{\varnothing_2}\}}$,
condition (\ref{right}) is equivalent to
\begin{equation*}
x(K) +x_{\recordnode{\{A_1\}}}+x_{\recordnode{\varnothing_1}}= 1,\ \ 
x(K) + x_{\recordnode{\varnothing_2}} + x_{\recordnode{\{A_2\}}} = 1,\ \ 
x(K) +x_{\recordnode{\{A_1\}}}+x_{\recordnode{\{A_2\}}}\leq 1,
\end{equation*}
so, with
\begin{equation}\label{amalproj-equiv}
x_{\recordnode{\{A_1\}}}=1-x(K)-x_{\recordnode{\varnothing_1}},\ \ 
x_{\recordnode{\{A_2\}}}=1-x(K)-x_{\recordnode{\varnothing_2}},\ \ 
x(K) + x_{\recordnode{\varnothing_1}} + x_{\recordnode{\varnothing_2}} \geq  1.
\end{equation}
We now eliminate $x_{\recordnode{\{A_1\}}}$ and $x_{\recordnode{\{A_2\}}}$.
In (\ref{left}), this amounts to deleting 
$\recordnode{\{A_1\}}$ from $G(K,V_1\cup A_1)$,
In (\ref{mid}), this amounts to deleting 
$\recordnode{\{A_2\}}$ from $G(K,V_2\cup A_2)$.
Since
$G(K,V_1\cup A_1)- \recordnode{\{A_1\}}$
and 
$G(K,V_2\cup A_2)- \recordnode{\{A_2\}}$
are the blocks of the amalgam decomposition of $G$, we get the following result.
\begin{theorem}\label{amalgam_extended_formulation} 
If $G_1$ and $G_2$ are the blocks of an amalgam decomposition of $G=(V,E)$
using the amalgam separation $(V_1,A_1,K,A_2,V_2)$,
then the stable set polytope $\pol{G}$ of $G$ satisfies:
\\[.5\baselineskip]
\mbox{}\hfill
$
\pol{G}=\{x_G\:x\in \R^{V\cup\{\recordnode{\varnothing_1},\recordnode{\varnothing_2}\}},\;
x_{G_1}\in \pol{G_1}, \;
x_{G_2}\in \pol{G_2},\;
x(K) +x_{{\recordnode{\varnothing_1}}}+x_{{\recordnode{\varnothing_2}}}\ge 1\}.
$
\hfill\mbox{}
\end{theorem}
\noindent
If we have original space descriptions for 
$\pol{G_1}$ and $\pol{G_2}$, Theorem~\ref{amalgam_extended_formulation} yields
an extended formulation for $\pol{G}$ with
$x_{\recordnode{\varnothing_1}}$ and $x_{\recordnode{\varnothing_2}}$
as the only extra variables.
With Fourier-Motzkin elimination it easy to remove
$x_{\recordnode{\varnothing_1}}$ and $x_{\recordnode{\varnothing_2}}$
from that extended formulation.
This leads to a new proof of the following result of
Burlet and Fonlupt (see \cite{NS} for an extension).
\begin{corollary}[Burlet and Fonlupt\cite{BF2}] \label{polycap}
Let the stable set polytopes of the blocks of an amalgam
decomposition of $G$ be described by the following systems:
\begin{equation}\label{burletfonlupt1}
x\geq 0,\ \ D^1x \leq \delta^1,\ \ x_{\recordnode{\varnothing_1}}\geq 0, \ \ 
\mbox{and}\ \ 
x_{\recordnode{\varnothing_1}} + c^{1,i}x \leq\gamma^{1,i}\ \ (i=1,\ldots,n_1),
\end{equation}
and 
\begin{equation}\label{burletfonlupt2}
x\geq  0,\ \ D^2x\leq \delta^2,\ \ x_{\recordnode{\varnothing_1}}\geq 0, \ \ 
\mbox{and}\ \ 
x_{\recordnode{\varnothing_2}} + c^{2,i}x\leq\gamma^{2,i}\ \ (i=1,\ldots,n_2)
\end{equation}
where $\recordnode{\varnothing_1}$ and $\recordnode{\varnothing_2}$ are
the nodes that are not in $G$.
Then $\pol{G}$ is given by the following system:
\begin{eqnarray}\label{sysblockamal}
x\geq 0,\ \
D^1x \leq  \delta^1,\ \ 
D^2 \leq  \delta^2,&&\label{burletfonlupt3}\\[.3\baselineskip]
\left[c^{1,i}+ c^{2,j}\right]x-x(K)\leq  \gamma^{1,i}+\gamma^{2,j}-1
&& (i=1,\ldots,n_1,\ j=1,\ldots,n_2).\label{burletfonlupt4}
\end{eqnarray}
where $K$ is the clique in the amalgam separation.
\end{corollary}
\begin{proof}
Let $G_1$ and $G_2$ be the blocks of the amalgam decomposition,
where (\ref{burletfonlupt1}) describes $\pol{G_1}$
and (\ref{burletfonlupt2}) describes $\pol{G_2}$.
By Theorem~\ref{amalgam_extended_formulation},
$\pol{G}$ consists of all $x$ for which there exists
$x_{\recordnode{\varnothing_1}}$ and
$x_{\recordnode{\varnothing_2}}$ such that
$(x, x_{\recordnode{\varnothing_1}},x_{\recordnode{\varnothing_2}})$
satisfies
 (\ref{burletfonlupt1}), (\ref{burletfonlupt2}), and
\begin{equation}\label{up}
x(K)+x_{\recordnode{\varnothing_1}}+x_{\recordnode{\varnothing_2}}\ge 1.
\end{equation}
Since  (\ref{burletfonlupt1}) describes $\pol{G_1}$,
we get that (\ref{burletfonlupt1})
implies ``$x(K)+x_{\recordnode{\varnothing_1}}\le 1$''.
Subtracting that inequality from
(\ref{up}),
yields ``$x_{\recordnode{\varnothing_2}}\geq 0$''.
In other words: the constraint ``$x_{\recordnode{\varnothing_2}}\geq 0$''
is redundant in the
system of linear inequalities given by
(\ref{burletfonlupt1}), (\ref{burletfonlupt2}) and (\ref{up}).
By symmetry, the same applies to ``$x_{\recordnode{\varnothing_1}}\geq 0$''.
So the system of linear inequalities given by
(\ref{burletfonlupt1}), (\ref{burletfonlupt2}) and (\ref{up}),
 is equivalent to the system consisting of
(\ref{burletfonlupt3}) together with:
\begin{eqnarray}
&x_{\recordnode{\varnothing_1}}&\le \gamma^{1,i} -c^{1,i}x\ \  (i=1,\ldots, n_1)\label{fm1}\\
&x_{\recordnode{\varnothing_2}}&\le \gamma^{2,j} -c^{2,j}x\ \  (j=1,\ldots, n_2)\label{fm2}\\
1-x(K)-x_{\recordnode{\varnothing_1}}\le& x_{\recordnode{\varnothing_2}}&\label{fm3}
\end{eqnarray}
Eliminating $x_{\recordnode{\varnothing_2}}$, replaces (\ref{fm1})-(\ref{fm3}) by
\begin{eqnarray}
&x_{\recordnode{\varnothing_1}}&\le \gamma^{1,i} -c^{1,i}x\ \  (i=1,\ldots, n_1)\label{fm4}\\
1-x(K)+ c^{2,j}x- \gamma^{2,j} \le& x_{\recordnode{\varnothing_1}}&
\hspace{2.25cm} (j=1,\ldots, n_2)\label{fm5}
\end{eqnarray}
Eliminating $x_{\recordnode{\varnothing_1}}$, replaces (\ref{fm4}) and (\ref{fm5}) by
(\ref{burletfonlupt4}).
\end{proof}
\subsection{Proof of Theorem 2}
\begin{lemma}\label{hole-template-pol}
Stable set polytopes of record graphs of 
fan-templates 
have compact extended formulations
that can be constructed in polynomial time.
\end{lemma}
\begin{proof} 
Let $H$ be the record graph of a fan-template with base $(u,c,v)$.
Then $\pol{H}$ is the convex hull of $\pol{H-c}$ and of
a face of the convex hull of the characteristic vector of $\{c\}$
and $\pol{H-\neighbor{H}{c}-c}$.
Since, by Lemma~\ref{zipfan}, the graphs $H-c$ and $H-\neighbor{H}{c}-c$ are decomposable
by 2-node clique sets into a list of at most $|V|$ graphs, each with at most 8 nodes,
the lemma follows from Corollary~\ref{cliquecut} and Balas's formula (\ref{balas}).
\end{proof}
\begin{theorem}
The stable set polytopes of cap-free odd-signable graphs have a compact
extended formulation that can be constructed in polynomial time.
\end{theorem}
\begin{proof}
When graph $G$ has as a node $u$ adjacent to all other nodes,
$\pol{G}$ is the convex hull of the characteristic vector of $\{u\}$ and $\pol{G-u}$.
Hence in that case it follows from (\ref{balas}), that
 $\pol{G}$ has an extended formulation with only three more variables
than any such formulation for $\pol{G-u}$.
Recall from Section~\ref{cliqueroot},
that clique cutset separations and amalgam separations are $1$-linearizable and that
a $1$-linearized decomposition-list of a graph $G=(V,E)$ can have at most $|V|^2$ members.
Hence, by Lemma \ref{hole-template-pol},
the result follows from the decomposition results Theorem \ref{capdec},
\ref{tf}  and
the polyhedral  composition results Corollary~\ref{cliquecut} and
Theorems \ref{th:templ-extform} and \ref{amalgam_extended_formulation}.
\end{proof}


\begin{thebibliography}{99}

\bibitem{B} E. Balas. Disjunctive programming: Properties of the convex hull of feasible points. {\it Discrete Applied Mathematics} {\bf 89} (1998) 3-14.
\bibitem {BF1} M. Burlet, J. Fonlupt. Polynomial algorithm to recognize a Meyniel Graph. In: C. Berge and V. Chv\'atal (eds.). {\it Topics on Perfect graphs.}
Annals of Discrete Mathematics 21, North Holland, 1984, pp. 225-252.
\bibitem {BF2} M. Burlet, J. Fonlupt. Polyhedral consequences of the amalgam operation. {\it Discrete Mathematics} {\bf 130} (1994) 39-55.
\bibitem{CS} M. Chudnovsky, P. Seymour. Claw-free graphs. V. Global structure. {\it Journal of Combinatorial Theory, Series B} {\bf 98} (2008) 1373-1410.
\bibitem{CHV} V. Chv\'atal. On certain polytopes associated with graphs. {\it Journal of Combinatorial Theory, Series B} {\bf 18} (1975) 138-154.
\bibitem{CAP} M. Conforti, G. Cornu\'{e}jols, A. Kapoor, K  Vu\v{s}kovi\'c. Even and odd holes in cap-free graphs.  {\it Journal of Graph Theory} {\bf 30} (1999) 289-308.
\bibitem{TFODD} M. Conforti, G. Cornu\'{e}jols, A. Kapoor, K. Vu\v{s}kovi\'c.  Triangle-free graphs that are signable without even holes. {\it Journal of Graph Theory} {\bf 34} (2000) 204-220.
\bibitem{OSdec} M. Conforti, G. Cornu\'{e}jols, A. Kapoor, K. Vu\v{s}kovi\'c.  Even-hole-free graphs, Part I: decomposition theorem.  {\it Journal of Graph Theory} {\bf 39} (2002) 6-49.
\bibitem{OSrec} M. Conforti, G. Cornu\'{e}jols, A. Kapoor, K. Vu\v{s}kovi\'c. Even-hole-free graphs, Part II: recognition algorithm.  {\it Journal of Graph Theory} {\bf 40} (2002) 238-266.
\bibitem{CC} G. Cornu\'{e}jols, W. H. Cunningham. Compositions for perfect graphs. {\it Discrete Mathematics} {\bf 55} (1985) 245-254.
\bibitem{FOG} Y. Faenza, G. Oriolo, G. Stauffer. Separating stable sets in claw-free graphs via Padberg-Rao and compact linear programs.
In: Y. Rabani (eds.), {\it Proceedings of the Twenty-Third Annual ACM-SIAM Symposium on Discrete Algorithms}, SIAM, 2012,  pp. 1298-1308.
\bibitem{GGV} A. Galluccio, C. Gentile, P. Ventura. 2-clique-bond of stable set polyhedra.
{\em Discrete Applied Mathematics} {\bf 161} (2013) 1988-2000.
\bibitem{GLS}  M. Gr\"otschel, L. Lov\'asz, A. Schrijver. The ellipsoid method and its consequences in combinatorial optimization. {\em Combinatorica} {\bf 1} (1981) 169-197.
\bibitem{NS} P. Nobili, A. Sassano. Polyhedral properties of clutter amalgam. {\it SIAM Journal on Discrete Mathematics} {\bf 6} (1993) 139-151.
\bibitem{WHI1} S. Whitesides. An algorithm for finding clique-cutsets. {\it Information Processing Letters} {\bf 12} (1981) 31-32.
\end{thebibliography}
\end{document}